\documentclass[11pt]{article}

\usepackage[margin=1.27in]{geometry}

\usepackage[misc]{ifsym} 
\usepackage{lipsum}
\usepackage{mathrsfs}
\usepackage{amssymb}
\usepackage{amsmath}
\usepackage{latexsym}
\usepackage{bm}
\usepackage{enumitem}
\usepackage{graphicx}
\usepackage{amscd}
\usepackage{extarrows}
\usepackage{amsfonts}
\usepackage{amssymb}
\usepackage{indentfirst}
\usepackage{bbm}
\usepackage{mathdots}
\usepackage{mathtools}
\usepackage[linktocpage]{hyperref}
\usepackage{amsthm}
\usepackage{ upgreek }
\usepackage{empheq}
\hypersetup{
    colorlinks,
    citecolor=black,
    filecolor=black,
    linkcolor=blue,
    urlcolor=black
}

\usepackage{quiver}
\usepackage[multiple]{footmisc}
\usepackage[title]{appendix}

\usepackage{float}

\theoremstyle{plain}
\newtheorem{theorem}{Theorem}[section]
\newtheorem{proposition}[theorem]{Proposition}
\newtheorem{definition}[theorem]{Definition}
\newtheorem{lemma}[theorem]{Lemma}
\newtheorem{corollary}[theorem]{Corollary}
\newtheorem*{theorem*}{Theorem}
\newtheorem*{lemma*}{Lemma}

\theoremstyle{remark}
\newtheorem{remark}[theorem]{Remark}

\theoremstyle{definition}
\newtheorem{example}[theorem]{Example}

\newcommand*\dd{\mathop{}\!\mathrm{d}}

\DeclareMathOperator{\End}{End}

\DeclareMathOperator{\Hom}{Hom}

\DeclareMathOperator{\Aut}{Aut}

\DeclareMathOperator{\Ind}{Ind}

\DeclareMathOperator{\Rep}{Rep}

\DeclareMathOperator{\Res}{Res}
\DeclareMathOperator{\Speh}{Speh}
\DeclareMathOperator{\pr}{pr}
\DeclareMathOperator{\Nrd}{Nrd}

\DeclareMathOperator{\SL}{SL}
\DeclareMathOperator{\GL}{GL}

\DeclareMathOperator{\JL}{JL}

\DeclareMathOperator{\tp}{temp}

\DeclareMathOperator{\cusp}{cusp}
\DeclareMathOperator{\disc}{disc}

\DeclareMathOperator{\irr}{Irr}

\DeclareMathOperator{\Sp}{Sp}

\DeclareMathOperator{\rk}{rank}
\DeclareMathOperator{\supp}{supp}
\DeclareMathOperator{\MW}{MW}

\title{Automorphic $C^*$-Algebras of Reductive Groups}
\author{Jun Yang}
\date{}

\begin{document}
\maketitle
\begin{abstract}
Let $G$ be a reductive group over a number field $F$.  
We define $C^*_{\mathcal A}(G(\mathbb A))$, the {\it automorphic $C^*$-algebra  of $G$}, to be the $C^*$-algebraic image of the full automorphic representation of $G(\mathbb A)$ on $L^2(G(F)\backslash G(\mathbb A))$.
Using the Langlands spectral decomposition with respect to discrete Levi data, we construct an injective $*$-homomorphism
\begin{equation}\label{emain0}
C^*_{\mathcal A}(G(\mathbb A))
\longrightarrow
\bigoplus_{[M,\sigma]}
\mathcal K_{C_0(\widehat{A_M})}
\bigl(\Ind_P^G C_0(\widehat{A_M},H_{M,\sigma})\bigr)^{W(G,M,\sigma)},
\end{equation}
where $[M,\sigma]$ ranges over the associate classes of discrete Levi data.
When $G$ is $\GL(n)$ or an inner form of it, we prove that the map \eqref{emain0} is an isomorphism of $C^*$-algebras.
\end{abstract}

\tableofcontents

\section{Introduction}\label{sintro}

One of the classical themes in harmonic analysis on reductive groups is to
encode the Plancherel decomposition, and more generally the topology of unitary representations, in the structure of the group $C^*$-algebra.  
If $G$ is a locally compact group, then the spectrum of the maximal group $C^*$-algebra $C^*(G)$ is naturally identified with the unitary
dual $\irr(G)$. 
The spectrum of the reduced group $C^*$-algebra $C_r^*(G)$ is the tempered dual $\irr_{\tp}(G)$, which consists of the irreducible unitary representations weakly contained in the regular representation.  
For a reductive group over a local field, the latter is precisely the natural $C^*$-algebraic object for Plancherel theory.  

This point of view has a long history for real reductive groups.  
Starting from Harish-Chandra's Plancherel theory \cite{HCharm1} and Arthur's description of the Harish-Chandra Schwartz algebra \cite{Arthur1983}, Wassermann obtained a remarkably simple description of the reduced $C^*$-algebra of a real reductive group, up to Morita equivalence, and used it in his approach to the Connes--Kasparov conjecture \cite{Wasserman1987}.
A detailed $C^*$-algebraic treatment was later given by Clare, Crisp and Higson \cite{CCH2016}, building also on the Hilbert-module realization of parabolic induction developed by Clare in \cite{Clare2013}. 
Their structure theorem (see \cite[Theorem~6.8]{CCH2016}) may be written in the form
\begin{equation}\label{eIntroRealCstar}
 C_r^*(G)
 \cong
 \bigoplus_{[P,\sigma]}
 \mathcal K_{C_0(\widehat A)}
 \bigl(\Ind_P^G\mathcal H_\sigma\bigr)^{W_\sigma},
 \qquad
 \mathcal H_\sigma=C_0(\widehat A,H_\sigma),
\end{equation}
where $P=MAN$ ranges over appropriate parabolic subgroups, $\sigma$ ranges over square-integrable representations of $M$, and $W_\sigma$ is the corresponding Weyl stabilizer.  
The Weyl-group action in \eqref{eIntroRealCstar} is implemented by normalized Knapp--Stein intertwining operators.  
Thus the Plancherel decomposition is reflected by a decomposition of $C_r^*(G)$ into fixed-point algebras of compact operators on Hilbert $C^*$-modules. 
Clare, Higson, Song and Tang \cite{CHST2024} recently gave detailed descriptions of the $C^*$-algebraic Morita equivalence by Wassermann and computed the Connes-Kasparov morphism. 

There is a parallel, but geometrically rather different, theory for reductive groups over $p$-adic fields.  
For $p$-adic $\GL(n)$, Plymen started with the point of view of Harish-Chandra and proved that the reduced $C^*$-algebra is strongly Morita equivalent to a commutative $C^*$-algebra determined by unramified unitary characters of Levi subgroups \cite{Plymen1987}.
He subsequently established a $C^*$-Plancherel decomposition for general reductive $p$-adic groups \cite{Plymen1990}, giving the analogue of  \eqref{eIntroRealCstar} in the $p$-adic setting. 
For $\GL(n)$, this picture was further refined using the Bernstein decomposition and extended quotients in \cite{Plymen2002}. 
Aubert and Plymen later gave an explicit Plancherel decomposition as well as the Bernstein decompositions in \cite{AubtPlm05}.

An important difference between the real and $p$-adic group representations is the geometry of the unramified unitary parameters.  
In the real case, the relevant parameter space is a vector group, whereas for a $p$-adic Levi subgroup the group of unramified unitary characters is a compact torus.  
In particular, the contractibility argument available in Wassermann's real case is no longer available in general.  
Nevertheless, a substantial part of the real picture survives.  
Recently, Afgoustidis and Aubert established versions of Wassermann's theorem for broad classes of components of the tempered dual of classical $p$-adic groups, under natural assumptions on Weyl stabilizers and $R$-groups \cite{AfgAub2022}. 
More recently, Clare and Crisp have given a general description of the Plancherel components of the reduced $C^*$-algebra of a reductive $p$-adic group, up to Morita equivalence with twisted crossed products,
and used this description to study their $K$-theory \cite{ClareCrisp2026}. 
These developments make clear that the $C^*$-algebraic structure of local reductive groups is closely tied to three basic ingredients:
parabolic induction, the discrete series of Levi subgroups, and normalized intertwining operators.

The purpose of the present paper is to investigate a global analogue of this picture.  
Let $G$ be a connected reductive group over a number field $F$.  
At the global level, however, the ordinary reduced group $C^*$-algebra $C_r^*\bigl(G(\mathbb A)\bigr)$ is no longer the object most naturally adapted to the global Langlands conjectures for $G$.  
Indeed, by definition, it is associated
with the regular representation of the locally compact group $G(\mathbb A)$ on $L^2(G(\mathbb A))$. 
Therefore it detects the tempered dual of $G(\mathbb A)$.  
Global Langlands theory, however, singles
out the automorphic representation
\begin{center}
   $R_G:G(\mathbb A)\longrightarrow
 U\left(L^2\bigl(G(F)\backslash G(\mathbb A)\bigr)\right)$ 
\end{center}
and its cuspidal subrepresentations.
The spectral decomposition of this representation contains both the discrete and the continuous automorphic spectra. 
In particular, it contains generally non-tempered automorphic representations.  
Although the Ramanujan conjecture predicts that cuspidal representations of $\GL(n)$ are tempered, there exist counterexamples for $\Sp(4)$ \cite{HoweShapiro1979}. 
Thus neither the reduced nor the full group $C^*$-algebra isolates the representation-theoretic support 
relevant to automorphic forms. 
Langlands functoriality also predicts correspondences
between automorphic representations associated with homomorphisms of
$L$-groups (see \cite{Cogdell2003} and \cite[\S 26]{Art05}).  

Consequently, if one seeks a global operator-algebraic object compatible with the Langlands philosophy, it is natural to retain precisely the part of the unitary dual that occurs in the
automorphic spectrum in the sense of weak containment.  
Motivated by this observation, we introduce the \emph{automorphic $C^*$-algebra} of $G$
\begin{equation*}
 C^*_{\mathcal A}\bigl(G(\mathbb A)\bigr)
 :=
 R_G\bigl(C^*(G(\mathbb A))\bigr)
 \cong
 C^*(G(\mathbb A))/\ker R_G.
\end{equation*}
Equivalently, the irreducible representations of
$C^*_{\mathcal A}(G(\mathbb A))$ are precisely automorphic representations, i.e., the irreducible unitary representations of $G(\mathbb A)$ which are weakly contained in $R_G$.

Our aim is to describe the structure of
$C^*_{\mathcal A}(G(\mathbb A))$ using Langlands' spectral
decomposition by discrete Levi data.  
Assume $F=\mathbb{Q}$ and let $P=MN$ be a standard
parabolic subgroup of $G$. 
Let $A_M$ be the maximal central subgroup of $M$ over $\mathbb{Q}$ which is a $\mathbb{Q}$-split torus. 
We simply write $A_M$ for $A_M(\mathbb R)^0$ and then identify $\widehat{A_M}$ with $i\mathfrak a_M^*$. 
Let $\sigma\in\Pi_{\disc}\bigl(M(\mathbb A)^1\bigr)$ be an irreducible discrete automorphic representation.  
We denote by
$H_{M,\sigma}$ the full $\sigma$-isotypic subspace of $L^2_{\disc}\bigl(M(F)\backslash M(\mathbb A)^1\bigr)$, which includes its automorphic multiplicity. 
Associated to $(M,\sigma)$, we have 
the Hilbert $C_0(\widehat{A_M})$-module
\begin{equation*}
 \mathcal E_{M,\sigma}
 :=
 \Ind_{P(\mathbb A)}^{G(\mathbb A)}
 \left(
 C_0(\widehat{A_M},H_{M,\sigma})
 \right).
\end{equation*}
Its fiber over $\lambda\in i\mathfrak a_M^*$ is the Hilbert space of
the induced automorphic representation
\[
 \Ind_{P(\mathbb A)}^{G(\mathbb A)}(\sigma_\lambda),
 \qquad
 \sigma_\lambda
 =
 \sigma\otimes
 e^{\langle\lambda,H_M(\,\cdot\,)\rangle}.
\]
For a Weyl group element $w$, the global normalized intertwining operators $M(w,\lambda)$ of Langlands and Arthur provide the symmetries of this family. 
If we set
\begin{center}
    $ W(G,M,\sigma):=\{w\in W(G,M):{}^w\sigma\simeq\sigma\}$,
\end{center}
then the scattering operators define twisted unitary automorphisms of
$\mathcal E_{M,\sigma}$ by
\begin{center}
$(U_wF)(\lambda):=
 M(w,w^{-1}\lambda)F(w^{-1}\lambda)$.  
\end{center}
Their exact functional equations imply
that $W(G,M,\sigma)$ acts by automorphisms on $\mathcal K_{C_0(\widehat{A_M})}
 \bigl(\mathcal E_{M,\sigma}\bigr)$. 
This is the global analogue of the Weyl-group actions occurring in the local $C^*$-Plancherel decompositions above. 
The underlying spectral decomposition is the Langlands decomposition by the discrete spectra of Levi subgroups, where we use the formulations of \cite[\S 7]{Art05} and \cite[\S VI.2]{MW1995}.

The main result of this paper is an operator-algebraic realization of the
Langlands spectral decomposition in terms of the discrete automorphic spectra
of Levi subgroups. 
\begin{theorem}\label{tmain}
Let $G$ be a connected reductive group over a number field.   
There is a natural injective $*$-homomorphism of $C^*$-algebras
\begin{equation}\label{eIntroMain}
    \Phi:
C^*_{\mathcal A}\bigl(G(\mathbb A)\bigr)
\longrightarrow
\bigoplus_{[M,\sigma]}
\mathcal K_{C_0(\widehat{A_M})}
\bigl(\Ind_{P(\mathbb A)}^{G(\mathbb A)}
C_0(\widehat{A_M},H_{M,\sigma})\bigr)^{W(G,M,\sigma)},
\end{equation}
where $[M,\sigma]$ runs through the associate classes of discrete Levi
data. 
\end{theorem}

We then apply the classification of automorphic representations of $\GL(n)$ and its inner forms by Moeglin--Waldspurger \cite{MW1989} and Badulescu-Renard \cite{BaRe10}. 

\begin{corollary}\label{cmain}
If $G=\GL(n)$ or an inner form of $\GL(n)$, the homomorphism $\Phi$ above is an isomorphism of $C^*$-algebras.   
\end{corollary}

We briefly describe the main ingredients of the proof.  The first step is to
translate parabolic induction into the language of Hilbert $C^*$-modules.
For every standard parabolic subgroup $P=MN$, we construct in Section \ref{sHilbertC} an automorphic Hilbert correspondence $X_P^{\mathcal A}$ from $C^*_{\mathcal A}(G(\mathbb A))$ to
$C^*_{\mathcal A}(M(\mathbb A))$. 
It is obtained from the usual Rieffel
induction correspondence by showing that the left action of
$C^*(G(\mathbb A))$ factors through the automorphic quotient.  Tensoring this
correspondence with the family
$C_0(\widehat{A_M},H_{M,\sigma})$ realizes, fiber by fiber, the family of
automorphically induced representations $\Ind_{P(\mathbb A)}^{G(\mathbb A)}(\sigma_\lambda)$ with $\lambda\in\widehat{A_M}\simeq i\mathfrak a_M^*$. 

The second step, carried out in Section \ref{scptCAG}, is to prove the compactness
properties needed to pass from the Hilbert-module realization to a
$C^*$-algebraic decomposition.  The automorphic $C^*$-algebra of a Levi
subgroup acts through compact operators on the modules associated with its
discrete automorphic spectrum, and after parabolic induction the same remains
true for the action of $C^*_{\mathcal A}(G(\mathbb A))$.  Combining this with
the orthogonal decomposition of the discrete spectrum into its
$\sigma$-isotypic subspaces gives compact operator-valued component maps. 

In Section \ref{sinjG}, we incorporate the global scattering operators.  
The normalized intertwining operators of Langlands and Arthur give, for $w\in W(G,M,\sigma)$, twisted unitary automorphisms
\begin{center}
    $(U_wF)(\lambda)
=
M(w,w^{-1}\lambda)F(w^{-1}\lambda)$
\end{center}
of $\mathcal E_{M,\sigma}$. 
Their functional equations imply the action of $W(G,M,\sigma)$ on $\mathcal K_{C_0(\widehat{A_M})} \bigl(\mathcal E_{M,\sigma}\bigr)$. 
The image of the automorphic $C^*$-algebra in each component is contained in the corresponding fixed-point algebra. 
With the Langlands spectral decomposition by discrete Levi data, we prove Theorem \ref{tmain}: $\Phi$ is an injective homomorphism (see Theorem \ref{tmain1}). 

In Section \ref{sisoGLn}, we establish the remaining surjectivity for $G=\GL(n)$.   
Here the global classification is sufficiently rigid to identify the irreducible representations of each fixed-point component with the automorphic representations arising from the corresponding discrete Levi datum. 
Two features are crucial. 
First, discrete automorphic representations of $\GL(m,\mathbb A)$ occur with multiplicity one.  
Second, the M\oe glin--Waldspurger description of the discrete spectrum in terms of Speh representations \cite{MW1989}, together with the Jacquet--Shalika theory \cite{JacShlk1981} and strong multiplicity one, gives the required uniqueness and disjointness of the automorphic induction data.  
Consequently distinct discrete Levi data give disjoint automorphic supports, while the Weyl group accounts exactly for the equivalences within a fixed datum.  
This allows the component maps to be shown surjective and hence proves that $\Phi$ is an isomorphism for $\GL(n)$ (see Corollary \ref{cisoGLn}). 

In Section \ref{sisoinnerGLn}, we apply the classification of automorphic representations of inner forms of $\GL(n)$. 
We start with the global Jacquet--Langlands correspondence for inner forms of $\GL(n)$ established by Badulescu--Renard \cite{BaRe10}. 
A similar theory of Speh representations is available in this setting and yields the corresponding disjointness property of the discrete Levi datum of the inner form. 
The isomorphism $\Phi$ for the inner form case then follows (see Corollary \ref{cisoInnerGLn}).  

We conclude with a brief outline of the paper. 
In Section \ref{sautodec}, we recall the Langlands spectral decomposition and refine it according to individual discrete Levi data.  
Section \ref{sHilbertC} constructs the Hilbert correspondences implementing automorphic parabolic induction.
In Section \ref{scptCAG}, we establish the compactness properties needed for the $C^*$-algebraic decomposition.  
In Section \ref{sinjG}, we construct the component homomorphisms, introduce the global Weyl actions through normalized
intertwining operators, and prove \eqref{eIntroMain}.
Section \ref{sisoGLn} specializes to $\GL(n)$ and uses the classification and uniqueness properties of its automorphic spectrum to prove that \eqref{eIntroMain} is an isomorphism. 
Finally, Section \ref{sisoinnerGLn} extends  the results to inner forms of $\GL(n)$.

\section{The Langlands decomposition by discrete Levi datum}\label{sautodec}

Let $G$ be a connected reductive group over a number field $F$. 
Let $G(\mathbb{A})$ be the group of adelic points which contains the group of $F$-rational points $G(F)$ as a discrete subgroup. 
We will apply the notations and assumptions from \cite[Part I]{Art05}.
For simplicity, we assume $F=\mathbb{Q}$ from now on. 

\begin{itemize}
\item $A_G:=$ the largest central subgroup of $G$ over $\mathbb{Q}$ which is a $\mathbb{Q}$-split torus, i.e., $A_G\cong_{\mathbb{Q}}\GL(1)^k$ for some $k\geq 0$. 
\item $X(G)_{\mathbb{Q}}:=\Hom_{\mathbb{Q}}(G,\GL(1))$, which is a free abelian group of rank $k=\rk A_G$. 
\item $\mathfrak{a}_{G}:=\Hom_{\mathbb{Z}}(X(G)_{\mathbb{Q}},\mathbb{R})$. We define a surjective homomorphism $H_G\colon G(\mathbb{A})\to\mathfrak{a}_{G}$ given by
\begin{equation}\label{eHG}
    \langle H_G(x),\chi\rangle=\log|\chi(x)|_{\mathbb{A}}
\end{equation}
for $x\in G(\mathbb{A})$ and $\chi\in X(G)_{\mathbb{Q}}$.

\item $G(\mathbb{A})^{1}:=\{x\in G(\mathbb{A})\mid H_{G}(x)=0\}=\bigcap_{\chi\in X(G)_{\mathbb{Q}}}\ker|\chi(\cdot)|_{\mathbb{A}}$. 
Note that 
\begin{center}
 $G(\mathbb{A})=G(\mathbb{A})^{1}\times A_G(\mathbb{R})^0$. 
\end{center} 
If $G$ is semisimple, we have $G(\mathbb{A})^{1}=G(\mathbb{A})$. 

\item $R_G$ = the representation of $G(\mathbb{A})$ on $L^2(G(\mathbb{Q})\backslash
     G(\mathbb{A}))$. 

\item $R_{G}^1$ = the representation of $G(\mathbb{A})$ on $L^2(G(\mathbb{Q})A_{G}(\mathbb{R})^0\backslash
     G(\mathbb{A}))$. 
Note that $A_{G}(\mathbb{R})^0$ acts trivially by $R_G^1$. 

\item For $\lambda\in \mathfrak{a}_{G,\mathbb{C}}^*$, $R_{G,\lambda}=e^{\lambda\cdot H_{G}(\cdot)}\cdot R_{G}^1$, whose underlying space is also $L^2(G(\mathbb{Q})A_{G}(\mathbb{R})^0\backslash
     G(\mathbb{A}))$. 
Note that $A_{G}(\mathbb{R})^0$ acts via the factor $e^{\lambda\cdot H_{G}(\cdot)}$. 

\item $R_{G,\disc}$ = the representation of $G(\mathbb{A})$ on the discrete spectrum $L^2_{\disc}(G(\mathbb{Q})A_{G}(\mathbb{R})^0\backslash G(\mathbb{A}))$, which is a subrepresentation of $R_G^1$.

\item $R_{G,\disc,\lambda}=e^{\lambda\cdot H_{G}(\cdot)}\cdot R_{G,\disc}$, the subrepresentation of $R_{G,\lambda}$ on its discrete spectrum. 

\item As $G(\mathbb{A})= G^{1}(\mathbb{A})\times A_G(\mathbb{R})^0$, we have
\begin{equation}\label{eL2GG1A}
    L^{2}\bigl(G(\mathbb{Q})\backslash G(\mathbb{A})\bigr)
\cong
L^{2}\bigl(G(\mathbb{Q})\backslash G^{1}(\mathbb{A})\bigr)
\widehat{\otimes}L^{2}(A_G(\mathbb{R})^0). 
\end{equation}
Since $L^{2}(A_G(\mathbb{R})^0)\cong\int_{\widehat {A_G(\mathbb{R})^0}}^{\oplus}\mathbb{C}_{\chi}\,d\chi$, we further obtain
\begin{equation}\label{eRGRGlambda}
R_G\cong
\int_{\widehat {A_G(\mathbb{R})^0}}^{\oplus}
\left(R_{G}^1\otimes\chi\right)\,d\chi
\cong \int_{i\mathfrak{a}_{G,\mathbb{R}}^*}^{\oplus}
R_{G,\lambda}\,d\lambda.
\end{equation}

\end{itemize}

\begin{definition}
An \textbf{automorphic representation} of $G(\mathbb{A})$ is defined to be an irreducible representation of $G(\mathbb{A})$ which is weakly contained in $R_G$. 
We denote the set of equivalence classes of automorphic representations by $\Pi_{\mathcal{A}}(G(\mathbb{A}))$, or simply $\Pi_{\mathcal{A}}(G)$. 

A \textbf{discrete automorphic representation} of $G(\mathbb{A})$ is defined to be an irreducible subrepresentation contained in $R_{G,\disc,\lambda}$ for some $\lambda \in i\mathfrak{a}_{G,\mathbb{R}}^*$. 
We denote the set of equivalence classes of discrete automorphic representations by $\Pi_{\disc}(G(\mathbb{A}))$, or simply $\Pi_{\disc}(G)$. 
\end{definition} 

Note that a discrete automorphic representation $\pi$ is an irreducible subrepresentation of $R_G$ on   $L^2(G(\mathbb{Q})\backslash G(\mathbb{A}))$ if and only if $G(\mathbb{A})=G(\mathbb{A})^1$. 
Otherwise, $L^2(G(\mathbb{Q})\backslash G(\mathbb{A}))$ has no irreducible subrepresentations and $\pi$ is only weakly contained in $R_G$ as a term in the direct integral over $\widehat{A_{G}(\mathbb{R})^0}$ in the sense of \eqref{eRGRGlambda}. 
Thus, we have
\begin{equation}\label{edualGG1}
    \Pi_{\mathcal{A}}(G)\simeq \Pi_{\mathcal{A}}(G^1)\times\widehat{A_G(\mathbb{R})^0} \text{ and } \Pi_{\disc}(G)\simeq \Pi_{\disc}(G^1)\times\widehat{A_G(\mathbb{R})^0}
\end{equation}

Let $P_0$ be a fixed minimal parabolic subgroup of $G$ over $\mathbb{Q}$. 
Let $P_0=M_0 N_0$ be a fixed Levi decomposition and $A_0$ be the central subgroup of $M_0$.  
For a standard parabolic subgroup $P$, we let $P=M_P N_P$ be its Levi decomposition with the unique Levi subgroup $M_P$ of $P$ that contains $M_0$. 
Let $A_P=A_{M_P}$ be the central subgroup of $M_P$ and $\mathfrak{a}_{P}=\mathfrak{a}_{M_P}$ be the real vector space given by the Lie algebra of $A_P$ (after tensoring $\mathbb{R}$). 

\begin{itemize}
\item $W=W(G,M_0)=N(A_0,G)/M_0$. 

\item $W(G,M_P):=\{w\in W\mid w(M_P)=M_P\}$.  

\item $H_{M,\disc}:=L^2_{\disc}
\bigl(M(\mathbb{Q})\backslash M(\mathbb{A})^1\bigr)\cong 
\bigoplus_{\sigma\in\Pi_{\disc}(M^1)}m(\sigma)\cdot H_\sigma$, where each $m(\sigma)$ is finite. 
This is the underlying space of $R_{M,\disc}$ and $R_{M,\disc,\lambda}$. 

\item $H_{M,\sigma}:=L^2_{\disc}
\bigl(M(\mathbb{Q})\backslash M(\mathbb{A})^1\bigr)[\sigma]=\oplus_{k=1}^{m(\sigma)}H_\sigma$, the $\sigma$-isotypic subspace of $H_{M,\disc}$.

\item $R_{M,\sigma,\lambda}:= R_{M,\disc,\lambda}|_{H_{M,\sigma}}$.

\item For $\lambda\in \mathfrak{a}_{P,\mathbb{C}}^*$ and $g\in G(\mathbb A)$, we denote the induced representation of $G(\mathbb A)$ by
\begin{center}
    $I_{P}(\lambda,g):=\Ind_{P(\mathbb A)}^{G(\mathbb A)}(R_{M,\disc,\lambda}\otimes 1_{N(\mathbb{A})})(g)$.
\end{center}
For $\sigma\in \Pi_{\disc}(M^1)$, we also define the subrepresentation of $I_{P}(\lambda,g)$: 
\begin{center}
   $I_{P,\sigma}(\lambda,g):=\Ind_{P(\mathbb A)}^{G(\mathbb A)}(R_{M,\sigma,\lambda}\otimes 1_{N(\mathbb{A})})(g)$. 
\end{center}

\item $H_P:=$  the underlying space of $I_{P}(\lambda,~\cdot~)$, which is independent of $\lambda\in i\mathfrak{a}_{P}^{*}$. 
It is known that $H_P$ can be realized as measurable functions
\begin{center}
    $\phi\colon N_{P}(\mathbb{A})M_{P}(\mathbb{Q})A_{P}(\mathbb{R})^0\backslash G(\mathbb{A})\to \mathbb{C}$
\end{center}
such that
\begin{itemize}
    \item for any $x\in G(\mathbb{A})$,  
    the function $\phi_{x}(m)=\phi(mx)$ for $m\in M_{P}(\mathbb{Q})\backslash M_{P}(\mathbb{A})^1$ belongs to $L^2_{\disc}(M_{P}(\mathbb{Q})\backslash M_{P}(\mathbb{A})^1)$;
    \item $\|\phi\|^2=\int_{K}\int_{ M_{P}(\mathbb{Q})\backslash M_{P}(\mathbb{A})^1} |\phi(mk)|^2 dm dk<\infty$. 
\end{itemize}

\item  $H_{P,\sigma}:=$ the underlying space of $I_{P,\sigma}(\lambda,~\cdot~)$ for $\sigma\in \Pi_{\disc}(M^1)$. 

\item For two standard parabolic subgroups $P,P'$ of $G$, we let $W(\mathfrak{a}_{P},\mathfrak{a}_{P'})$ be the Weyl set of linear isomorphisms from $\mathfrak{a}_{P}$ onto $\mathfrak{a}_{P'}$ obtained by restriction of elements of the Weyl group $W$. 
The parabolic subgroups $P,P'$ are called {\it associated} if $W(\mathfrak{a}_{P},\mathfrak{a}_{P'})$ is non-empty. 

\item For $s\in W(\mathfrak{a}_{P},\mathfrak{a}_{P'})$ and $\lambda\in \mathfrak{a}_{P,\mathbb{C}}^*$, we define the operator
\begin{center}
    $M(s,\lambda)\colon H_P\to H_{P'}$
\end{center}
given by
\begin{equation}\label{einteropdef}
    \bigl(M(w_s,\lambda)\phi\bigr)(x)=\int\phi(w_s^{-1}nx)e^{(\lambda+\rho_P)(H_{P}(w_s^{-1}nx))}e^{(-s\lambda+\rho_{P'})(H_{P'}(x))},
\end{equation}
where the integral is taken over the quotient $(N_{P'}(\mathbb{A})\cap wN_{P}(\mathbb{A})w^{-1})\backslash N_{P'}(\mathbb{A})$ and $w_s$ is a representative of $s$ in $G(\mathbb{Q})$.

\item If $P=P'$, we have $W(\mathfrak{a}_{P},\mathfrak{a}_{P'})=W(G,M)$ and we may denote the above operator by $M(w,\lambda)$.
\end{itemize}

We have the following results for the operator $M(s,\lambda)$ (see \cite[\S 7]{Art05}). 

\begin{lemma}\label{lM}
For $s,t\in W(\mathfrak{a}_{P},\mathfrak{a}_{P'})$ and $\lambda\in \mathfrak{a}_{P,\mathbb{C}}^*$, we have
\begin{enumerate}
    \item $M(s,\lambda)$ intertwines the action of $I_{P}(\lambda,-)$ and $I_{P'}(s\lambda,-)$ of $G(\mathbb{A})$, i.e., 
\begin{center}
    $M(s,\lambda)I_{P}(\lambda,g)=I_{P'}(s\lambda,g)M(s,\lambda)$
\end{center}
for $g\in G(\mathbb{A})$. 

\item $M(ts,\lambda)=M(t,s\lambda)M(s,\lambda)$.

\item If $\lambda\in i\mathfrak{a}_{P}^*$, $M(s,\lambda)$ is analytic and extends to a unitary operator from $H_P$ to $H_{P'}$. 
\end{enumerate}
\end{lemma}

For an associate class $\mathcal{P}$ of standard parabolic subgroups of $G$, we define $\widehat{L}_{\mathcal{P}}$ to be the Hilbert space of families of measurable functions 
\begin{center}
    $F=\{F_{P}\colon i\mathfrak{a}_{P}^{*}\to H_P, P\in\mathcal{P}\}$ 
\end{center}
such that
\begin{enumerate}
    \item $F_{P'}(s\lambda)=M(s,\lambda)F_{P}(\lambda)$ for $s\in W(\mathfrak{a}_{P},\mathfrak{a}_{P'})$;
    \item $\|F\|^2:=\sum_{P\in\mathcal{P}}n_{P}^{-1}\int_{i\mathfrak{a}_{P}^{*}}\|F_{P}(\lambda)\|^2\dd \lambda<\infty$, where $n_P:=\sum_{P'\in\mathcal{P}}|W(\mathfrak{a}_{P},\mathfrak{a}_{P'})|$.
\end{enumerate}
Consider the map given by Eisenstein series: 
\begin{equation}\label{eFtoEisenstein}
F\mapsto \sum_{p\in\mathcal{P}}n_{P}^{-1}\int_{i\mathfrak{a}_{P}^{*}}E(x,F_P(\lambda),\lambda)\dd \lambda,
\end{equation}
which is defined when $F_{P}(\lambda)$ is a smooth, compactly supported function of  $\lambda$ with values in a finite-dimensional subspace of $\mathcal{H}_{P}^{0}$ (= $K$-finite vectors in $\mathcal{H}_{P}$ for a suitable maximal compact subgroup $K$ of $G(\mathbb{A})$).  
The map \eqref{eFtoEisenstein} extends to a map from $\widehat{L}_{\mathcal{P}}$ onto a closed $G(\mathbb{A})$-invariant subspace $L^2_{\mathcal{P}}(G(\mathbb{Q})\backslash G(\mathbb{A}))$. 
We get the decomposition 
\begin{equation}\label{eautodecP}
    L^2(G(\mathbb{Q})\backslash G(\mathbb{A}))\cong \bigoplus_{\mathcal{P}}L^2_{\mathcal{P}}(G(\mathbb{Q})\backslash G(\mathbb{A})), 
\end{equation}
which is indexed by the associate classes of standard parabolic subgroups (see \cite[Theorem 7.2]{Art05}). 

We recall a theorem by Langlands for automorphic representations of reductive groups (see \cite[Supplement Proposition 2]{BJ}). 
\begin{theorem}\label{tLldspara}
Let $G$ be a reductive group over a number field $F$. 
An irreducible representation $\pi$ of $G(\mathbb{A})$ is automorphic if and only if $\pi$ is an irreducible subquotient of $\Ind_{P(\mathbb{A})}^{G(\mathbb{A})}\delta$ with $\delta\in\Pi_{\cusp}(M(\mathbb{A}))$ for the Levi factor $M$ of a parabolic subgroup $P$ of $G$.  
\end{theorem}

\subsection{The decomposition via full discrete Levi spectrum}\label{ssdecfulldisc}

For an associate class $\mathcal{P}$ of standard parabolic subgroups of $G$, we take one representative $P=M_{P}N_{P}\in\mathcal{P}$ and set 
\begin{center}
   $W_{\mathcal P}=
W(\mathfrak{a}_{P},\mathfrak{a}_{P})$,
\end{center}
which is isomorphic to $W(G,M_P)$. 
Then we define
\begin{center}
    $I_{P}(\lambda)
=
\Ind_{P(\mathbb{A})}^{G(\mathbb{A})}
\left(
1_{N_P(\mathbb{A})}
\otimes
R_{M_P,\disc,\lambda}
\right),
\qquad
\lambda\in i\mathfrak{a}_{P}^*$, 
\end{center}
which acts on the Hilbert space $\mathcal{H}_{P}$. 
The space $\widehat L_{\mathcal P}$ above can then be rewritten as
\begin{center}
    $\widehat L_{\mathcal P}
\cong
\left[
\int_{i\mathfrak{a}_{P}^*}^{\oplus}
\mathcal{H}_{P}\,d\lambda
\right]^{W_{\mathcal P}}$.
\end{center}
It may be realized as the Hilbert space of measurable functions 
\begin{center}
    $F\colon i\mathfrak{a}_{P}^*\to \mathcal{H}_P$
\end{center}
such that
\begin{enumerate}
    \item $F(w\lambda)=M(w,\lambda)F(\lambda)$ for all $w\in W_{\mathcal P}$;
    \item $\|F\|^2:=\frac{1}{|W(G,M_P)|}\int_{i\mathfrak{a}_{P}^{*}}\|F(\lambda)\|^2\dd \lambda<\infty$. 
\end{enumerate}
Note that the action of $G(\mathbb{A})$ on this space is given by $
g\cdot F(\lambda)
=I_{P}(\lambda,g)F(\lambda)
$. 
Thus \eqref{eautodecP} can be rewritten as
\begin{equation}\label{eautodecPara}
    L^2\bigl(G(\mathbb Q)\backslash G(\mathbb{A})\bigr)
\cong
\bigoplus_{[P]}
\left[
\int_{i\mathfrak{a}_P^*}^{\oplus}
I_P(\lambda)\,d\lambda
\right]^{W_P}
\end{equation}
where $[P]$ runs over the representatives of the associate classes of standard parabolic subgroups,
$W_P=W(\mathfrak{a}_P,\mathfrak{a}_P)$
and
$I_P(\lambda)=\Ind_{P(\mathbb{A})}^{G(\mathbb{A})}
\left(
1_{N_P(\mathbb{A})}
\otimes
R_{M_P,\disc,\lambda}
\right)
$. 

Note that the class $P=G$ gives the discrete spectrum if $G(\mathbb{A})=G(\mathbb{A})^1$. 
Therefore the continuous
part is
$L^2_{\mathrm{cont}}
\bigl(G(\mathbb Q)\backslash G(\mathbb{A})\bigr)
\cong
\bigoplus_{\substack{[P]\\P\neq G}}
\left[
\int_{i\mathfrak{a}_P^*}^{\oplus}
I_P(\lambda)\,d\lambda
\right]^{W_P}.
$

\subsection{The decomposition via the discrete Levi datum}\label{ssdecdiscLevi}

Let $M,M'$ be Levi subgroups of $G$ and $\sigma\in \Pi_{\disc}(M(\mathbb{A})^1), \sigma'\in \Pi_{\disc}(M'(\mathbb{A})^1)$. 
Two pairs $(M,\sigma)$ and $(M',\sigma')$ are associated if there exists a
Weyl element $w$ such that
\begin{equation}\label{e1asso}
    wM=M',
\qquad
w\sigma\simeq \sigma'.
\end{equation}
It is known that two standard parabolic subgroups $P,P'$ are associated in the sense above if and only if their Levi subgroups $M,M'$ are conjugate (see \cite[Lemma 10.2.1]{Getz300}). 
In this way, \eqref{eautodecPara} can be re-written as
\begin{equation}\label{edecfull}
L^2\bigl(G(\mathbb{Q})\backslash G(\mathbb{A})\bigr)
\cong
\widehat{\bigoplus}_{[M]}
\left[
\int_{i\mathfrak{a}_M^{*}}^\oplus
I_{P_M}^G\bigl(R_{M,\disc,\lambda}\bigr)\,d\lambda
\right]_{}^{W(G,M)},
\end{equation}
where $[M]$ runs over associate classes of Levi subgroups and $P_M$ is a standard parabolic subgroup with its Levi factor $M$.

Let $\mathfrak{D}(G)$ denote the set of associate classes of such pairs, which is called the set of {\it discrete Levi data}. 
We fix a Levi subgroup $M=M_P$ of a parabolic subgroup $P$ of $G$. 
For
$\sigma\in\Pi_{\disc}(M^1)$, let
$H_{\sigma}$ be the Hilbert space of $\sigma$, and define
its multiplicity space in the discrete automorphic spectrum by
\begin{center}$
\mathcal M_{M}(\sigma)
:=
\operatorname{Hom}_{M(\mathbb{A})}
\bigl(
H_{\sigma},
L^2_{\disc}(M_P(\mathbb{Q})A_{M_P}(\mathbb{R})^0\backslash
     M_{P}(\mathbb{A}))
\bigr).
$\end{center}
It is finite-dimensional, and $m_{M}(\sigma)
:=
\dim \mathcal M_{M}(\sigma)
$. 
We let
\begin{center}
    $W(G,M,\sigma)
:=
\left\{
w\in W(G,M):
{}^{w}\sigma\simeq\sigma
\right\}$.
\end{center}
For $w\in W(G,M)$, we define the normalized intertwining operator on $\mathcal{E}$ as
\begin{equation}\label{enormintwop}
    (U_wF)(\lambda)
=
M(w,w^{-1}\lambda)F(w^{-1}\lambda).
\end{equation}
Then we have a finer decomposition for \eqref{edecfull}. 

\begin{proposition}
We have the following decomposition of $G(\mathbb{A})$-modules:
\begin{equation}\label{edecMdisc}
    L^{2}\bigl(G(\mathbb{Q})\backslash G(\mathbb{A})\bigr)
\cong
\widehat{\bigoplus}_{[M,\sigma]\in\mathfrak{D}(G)}
\left[
\int_{i\mathfrak{a}_{M}^{*}}^{\oplus}
\mathcal M_{M}(\sigma)
\widehat{\otimes}
\Ind_{P(\mathbb{A})}^{G(\mathbb{A})}
(\sigma_{\lambda})
\,d\lambda
\right]^{W(G,M,\sigma)}
\end{equation}
where the action of $W(G,M,\sigma)$ on the direct integral is the
unitary action defined by the normalized global intertwining
operators. 
\end{proposition}
\begin{proof}
Given $\lambda \in i\mathfrak{a}_M^*$ and $\sigma\in \Pi_{\disc}(M^1)$, we have $M(w,\lambda)H_{P,\sigma}=H_{P,{}^w\sigma}$ by \eqref{einteropdef} and the realization of $H_P$ above (see also \cite[\S 15, page 86]{Art05}). 
We simply write $W(G,M), W(G,M,\sigma)$ as $W,W_{\sigma}$. 
Set
\begin{center}
 $\mathcal E
=
\int_{i\mathfrak a_M^{*}}^\oplus H_P\,d\lambda$, and $\mathcal E_\sigma
=
\int_{i\mathfrak a_M^{*}}^\oplus
H_{P,\sigma}\,d\lambda$. 
\end{center}
Then we know that
\begin{center}
    $\mathcal E=\widehat{\bigoplus}_\sigma \mathcal E_\sigma$. 
\end{center}
Then the action of $W$ on $\mathcal{E}$ by the normalized intertwining operators $U_w$ gives  $U_w\mathcal{E}_\sigma=\mathcal{E}_{{}^w\sigma}$.

Let $\mathcal O=W\cdot\sigma$ be a $W$-orbit in
$\Pi_{\disc}(M(\mathbb A)^1)$ and set
\[
\mathcal E_{\mathcal O}
=
\widehat{\bigoplus}_{\tau\in\mathcal O}\mathcal E_\tau, 
\]
which is $W$-invariant. 
Hence $\mathcal{E}^W=
\widehat{\bigoplus}_{\mathcal O}
\mathcal{E}_{\mathcal O}^W$. 
Consider the orthogonal projection
\begin{center}
    $\pr_{\mathcal{O},\sigma}\colon \mathcal{E}_{\mathcal{O}}\to \mathcal{E}_{\sigma}$. 
\end{center}
For $F\in\mathcal E_{\mathcal O}^W$, 
$f:=\pr_\sigma F$ is $W_\sigma$-invariant, since $U_w$ preserves $\mathcal{E}_\sigma$ for $w\in W_\sigma$.  
Thus the restriction gives the map
\begin{center}
$\Res_{\mathcal{O},\sigma}:=\pr_{\mathcal{O},\sigma}|_{\mathcal{E}_{\mathcal O}^W}\colon 
\mathcal{E}_{\mathcal O}^W
\longrightarrow
\mathcal E_\sigma^{W_\sigma}$. 
\end{center}

Conversely, we let $f\in\mathcal E_\sigma^{W_\sigma}$ and choose a set
$\mathcal R$ of representatives for the finite coset $W/W_\sigma$.  
Define
\begin{equation}\label{eextendfromstabilizer}
F
=
\sum_{r\in\mathcal R}U_r f.
\end{equation}
The terms in this sum belong to the mutually orthogonal subspaces
$\mathcal E_{{}^r\sigma}$.  The definition is independent of the
choice of representatives: if $r'=rh$ with $h\in W_\sigma$, then $U_{r'}f=U_rU_hf=U_rf$. 
Moreover, for $w\in W$, the left multiplication by $w$ permutes the
cosets in $W/W_\sigma$ and hence $U_wF=F$. 
Thus $F\in\mathcal E_{\mathcal O}^W$ and
$\pr_{\mathcal{O},\sigma} F=f$.  
This proves that
$\Res_{\mathcal{O},\sigma}$ is bijective.

We equip $F\in \mathcal{E}_{\mathcal{O}}^{W}$ and $f\in \mathcal{E}_{\sigma}^{W_{\sigma}}$ with the normalized norms
\begin{center}
    $\|F\|^2:=
\frac1{|W|}
\int_{i\mathfrak a_M^{*}}\|F(\lambda)\|^2\,d\lambda$ and 
$\|f\|^2
=
\frac1{|W_\sigma|}
\int_{i\mathfrak a_M^{*}}\|f(\lambda)\|^2\,d\lambda$. 
\end{center}
Then $\Res_{\mathcal{O},\sigma}$ is a $G(\mathbb{A})$-equivariant unitary. 
Since $H_{P,\sigma,\lambda}
\cong
\mathcal M_M(\sigma)\widehat\otimes
\Ind_{P(\mathbb A)}^{G(\mathbb A)}(\sigma_\lambda)$, we obtain
\begin{center}
    $\left[
\int_{i\mathfrak a_M^{*}}^\oplus
H_P\,d\lambda
\right]^{W(G,M)}
\cong
\widehat{\bigoplus}_{[\sigma]}
\left[
\int_{i\mathfrak a_M^{*}}^\oplus
\mathcal M_M(\sigma)\widehat\otimes
\Ind_{P(\mathbb A)}^{G(\mathbb A)}
(\sigma_\lambda)
\,d\lambda
\right]^{W(G,M,\sigma)}$
\end{center}
where $[\sigma]$ runs over the $W(G,M)$-orbits in $\Pi_{\disc}(M(\mathbb A)^1)$. 
Then the proposition follows from the decomposition \eqref{edecfull}. 
\end{proof}

\section{The Hilbert correspondence for the Langlands spectral decomposition}\label{sHilbertC}

In this section, we construct a Hilbert correspondence which realizes the Langlands spectral decomposition with respect to the automorphic spectrum of Levi subgroups. 

Let us start with a locally compact group $G$. 
Let $\mu_G$ be a left-invariant Haar measure on $G$. 
For every unitary representation $(\pi,H_{\pi})$ of $G$, we can associate a representation of the Banach $*$-algebra $L^1(G,\mu_G)$, denoted by $\pi$, which is defined by
\begin{center}
    $\pi(f)=\int_{G}f(g)\pi(g)d\mu_{G}(g)$
\end{center}
for $f\in L^1(G,\mu_G)$. 
We can check that $\pi(f^*)=\pi(f)^*$ so that
\begin{center}
    $\pi\colon L^1(G,\mu_G)\to B(H_{\pi})$ 
\end{center}
is a $*$-representation. 
We define $C^*_{\pi}(G)$ be the $C^*$-subalgebra of $B(H_{\pi})$ generated by $\pi(f)$ with $f\in L^1(G)$. 

Let $\pi_{\rm univ}$ be the {\it universal representation} of $G$, which is defined to be the direct sum of all cyclic unitary representations\footnote{A unitary representation $(\pi,H_{\pi})$ of $G$ is called {\it cyclic} if there exists $v\in H_{\pi}$ such that the linear span $\pi(G)v$ is dense in $H_{\pi}$. } 
of $G$. 
Thus we have a $*$-representation $\pi_{\rm univ}\colon L^1(G)\to L(H_{\rm univ})$ and we defined the {\it maximal norm} for $f\in L^1(G)$ by $\|f\|_{\max}=\|\pi_{\rm univ}(f)\|$. 
The {\it maximal $C^*$-algebra} of $G$, denoted by $C^*(G)$, is defined to be $C^*_{\pi_{\rm univ}}(G)$, which is the completion of $L^1(G)$ with respect to $\|\cdot\|_{\max}$.

\begin{definition}
The \textbf{automorphic $C^*$-algebra} of a reductive group over a number field $F$ is defined to be the image of the full automorphic representation
\begin{center}
    $R_G\colon C^{*}(G(\mathbb{A}))\to B(L^2(G(F)\backslash G(\mathbb{A})))$, 
\end{center}
which will be denoted by $C^*_{\mathcal{A}}(G(\mathbb{A}))$, or simply $C^*_{\mathcal{A}}(G)$. 
\end{definition} 
For a parabolic subgroup $P$ of $G$ and its Levi decomposition $P=M_{P}N_{P}=MN$, we consider the automorphic representation $R_M$ of $M(\mathbb{A})$ and define the $C^*$-algebra $C^*_{\mathcal{A}}(M(\mathbb{A}))$ similarly. 
 
For simplicity, we will also take $F=\mathbb{Q}$ from now on. 
The goal of this section is to construct a Hilbert correspondence $X_{P}^{\mathcal{A}}$ between $C^*_{\mathcal{A}}(G(\mathbb{A}))$ and $C^*_{\mathcal{A}}(M(\mathbb{A}))$ such that for any irreducible representation $(\sigma,H)$ in the discrete spectrum $L^2_{\rm disc}(M(\mathbb{Q})\backslash M(\mathbb{A}))$ of $L^2(M(\mathbb{Q})\backslash M(\mathbb{A}))$, we have
\begin{center}
$\Ind_{P(\mathbb{A})}^{G(\mathbb{A})}\sigma\cong X_{P}^{\mathcal{A}}\otimes_{C^*_{\mathcal{A}}(M(\mathbb{A}))}H$. 
\end{center}
The induction functor $\Ind_{C^{*}_{\mathcal{A}}(M(\mathbb{A}))}^{C^{*}_{\mathcal{A}}(G(\mathbb{A}))}(X_{P}^{\mathcal{A}},-)$ relates automorphic representations of $G$ to those of its Levi subgroups.

\subsection{Rieffel's induction module}

We briefly recall the necessary facts about group $C^*$-algebras, Hilbert $C^*$-modules and Rieffel's induction module (see \cite{Lance1995} and \cite{RWMoritaTrace1998}). 
Let $A,B$ be $C^*$-algebras. 

\begin{itemize}
    \item A {\it Hilbert $C^*$-module} $X$ over $B$, or simply a {\it Hilbert $B$-module $X$}, is a right $B$-module with a right $\mathbb{C}$-linear $B$-valued inner product
    \begin{center}
        $\langle~,~\rangle_B\colon X\times X\to B$
    \end{center}
    such that for any $x,y\in X$ and $b\in B$, we have
    \begin{enumerate}
        \item $\langle x,yb\rangle_B=\langle x,y\rangle_B\cdot b$; 
        \item $\langle x,y\rangle_B=\langle y,x\rangle_B^*$;
        \item $\langle x,x\rangle_B\geq 0$;
        \item $X$ is complete with respect to the norm $\|x\|^2=\|\langle x,x\rangle_B\|$
    \end{enumerate}

    \item The {\it $*$-algebra of adjointable operators} of a Hilbert $C^*$-module $X$ over $B$ is
    \begin{center}
        $\mathcal{L}_{B}(X)=\{T\colon X\to X\mid \exists T^*\colon X\to X, \langle Tx_1,x_2\rangle_{B}=\langle x_1,T^{*}x_2\rangle_{B}\}$,
    \end{center}
    which is a $C^*$-algebra with respect to the operator norm (see \cite[\S 1, page 8]{Lance1995}). 
    It is known that $\mathcal{L}_{B}(B)\cong M(B)$, the multiplier algebra of $B$ (see \cite[\S 2, page 15]{Lance1995}). 

    \item The {\it $*$-algebra of compact operators} of a Hilbert $C^*$-module $X$ over $B$, denoted by $\mathcal{K}_{B}(X)$, is the algebra generated by rank-one operators
    \begin{center}
        $T_{x_1,x_2}\colon x\mapsto x_1\langle x_2,x\rangle_B$
    \end{center}
    which is a $C^*$-algebra with respect to the operator norm (see \cite[\S 1, page 9]{Lance1995}). 
    Note that $\mathcal{K}_{B}(X)$ is not the algebra of usual compact operators unless $B\cong \mathbb{C}$. 
    We also know that $\mathcal{L}_{B}(X)=M(\mathcal{K}_{B}(X))$ (see \cite[Theorem 2.4]{Lance1995}). 

    \item A {\it Hilbert correspondence} for $(A,B)$ is a right Hilbert $C^*$-module $X$ over $B$ together with a $*$-homomorphism
    \begin{center}
        $\alpha_X\colon A\to \mathcal{L}_{B}(X)$. 
    \end{center}

    \item Given an $(A,B)$-Hilbert correspondence $X$ and a nondegenerate representation $\pi\colon B\to B(V_{\pi})$, we have an induced representation of $A$ on $X\otimes_{B}V_{\pi}$ via $X$ given by
    \begin{equation}\label{eCind}
        a(x\otimes v)=\alpha_{X}(a)x\otimes v
    \end{equation}
    for $a\in A$. 
    We denote this $A$-representation by $\Ind_{B}^{A}(X,\pi)$. 
    This induction gives us a functor
    \begin{center}
        $\Ind_{B}^{A}(X,-)\colon \Rep(B)\to \Rep(A)$ by $\pi\mapsto \Ind^{A}_{B}(X,\pi)$. 
    \end{center}

    \item 
    Let $\mathcal{E}$ be a Hilbert $A$-module and $X$ be an $(A,B)$-Hilbert correspondence. 
    \begin{enumerate}
        \item Let $\mathcal{E}\otimes ^{\rm alg} X$ be the algebraic tensor product. 
        \item Let $I$ be the subspace generated by the elements of the form $ea\otimes x-e\otimes \alpha_X(a)x$ with $e\in \mathcal{E}$ and $x\in X$. 
        \item The $B$-valued inner product is given by 
        \begin{center}
            $\langle e_1\otimes x_1,e_2\otimes x_2\rangle_B:=\langle  x_1,\alpha_{X}(\langle e_1,e_2\rangle_A) x_2\rangle_B$ 
        \end{center}
        for $e_1,e_2\in \mathcal{E}$ and $x_1,x_2\in X$. 
        \item The {\it interior tensor product} $\mathcal{E}\otimes_{A} X$ is the completion of $\bigl(\mathcal{E}\otimes ^{\rm alg} X\bigr)/I$ with respect to the $B$-valued inner product above (see \cite[Proposition 4.5]{Lance1995}). 
    \end{enumerate}
\item Let $\mathcal E$ be a Hilbert $A$-module and let $\alpha\in\Aut(A)$. 
An \emph{$\alpha$-twisted unitary automorphism}
of $\mathcal E$ is a complex-linear automorphism $U\colon \mathcal{E}\longrightarrow \mathcal{E}$
such that
\begin{equation}\label{etwistcover}
   U(\xi x)=U(\xi)\alpha(x),
\quad \langle U\xi,U\eta\rangle_A
=
\alpha\bigl(\langle\xi,\eta\rangle_A\bigr)
\end{equation}
for $\xi,\eta\in\mathcal{E}$ and $x\in A$. 
We also say that $U$ \emph{covers} the automorphism $\alpha$ of $A$.
\end{itemize}

The following lemma will be applied several times. See \cite[Proposition 4.7]{Lance1995} for the proof. 
\begin{lemma}\label{lcpttensor}
Let $X$ be a Hilbert $A$-module and $Y$ be a Hilbert correspondence for $(A,B)$. 
Suppose $A$ acts on $Y$ by compact operators, i.e., $\phi\colon A\to \mathcal{K}_{B}(Y)$ is the $*$-homomorphism.  
Then $\mathcal{K}_{A}(X)\subset \mathcal{K}_{B}(X\otimes_{A}Y)$. 
\end{lemma}

Let $G$ be a locally compact group and let $H$ be a closed subgroup of $G$.
Let $C^*(G)$ and $C^*(H)$ denote the full group $C^*$-algebras of $G$ and $H$,
respectively.  
Rieffel \cite{Rieffel1974} constructed a 
$(C^*(G),C^*(H))$-Hilbert correspondence, now usually called the {\it Rieffel  induction module}, which we denote by $\mathcal{E}_H^G$, or simply $\mathcal{E}$. 
Rieffel's induction module is the Hilbert $C^*$-module over $C^*(H)$ obtained by completing the space
$C_c(G)$ with respect to a $C^*(H)$-valued inner product. 
The $C^*(H)$-valued inner product is defined by
\begin{center}$
    \langle f_1,f_2\rangle_{C^*(H)}(h)
    =
    \int_G
    \overline{f_1(x)}\,f_2(xh)\,dx 
$\end{center}
for $f_1,f_2\in C_c(G)$, where the Haar measures are chosen compatibly and the usual modular-function corrections are included in the non-unimodular case. 
The right action of
$C_c(H)$ on $C_c(G)$ is given by
\begin{center}$
    (f\cdot a)(x)
    =
    \int_H f(xh)a(h^{-1})\,dh,
$\end{center}
for $f\in C_c(G)$, $a\in C_c(H)$, and $x\in G$.
The left action of $C^*(G)$ on $C_c(G)$ is given by convolution:
\begin{center}$
    (b\cdot f)(x)
    =
    \int_G b(y)f(y^{-1}x)\,dy ,
$\end{center}
for $b\in C_c(G)$.  This action extends to adjointable operators on $\mathcal{E}_H^G$:
\begin{center}$
    \alpha_{\mathcal{E}}\colon C^*(G)\longrightarrow \mathcal{L}_{C^*(H)}(\mathcal{E}_H^G). 
$\end{center}
Thus $\mathcal{E}_H^G$ is a
$\bigl(C^*(G),C^*(H)\bigr)$-Hilbert correspondence.

Given a nondegenerate representation 
$\pi\colon H\longrightarrow B(V_\pi)$, 
Rieffel's induction module defines a functor (see \eqref{eCind})
\begin{center}
    $\Ind_{C^*(H)}^{C^*(G)}(\mathcal{E},- )\colon \Rep(C^*(H))\to \Rep(C^*(G))$,
\end{center}
or simply $\Ind_{H}^{G}(\mathcal{E},-)$, between unitary representations $\Rep_{u}(H)$ and $\Rep_{u}(G)$. 

We construct the Hilbert correspondence for automorphic representations.  
Let $X_{P}^{G}$ be Rieffel's induction module, which is a $\bigl(C^{*}(G(\mathbb{A})),C^{*}(P(\mathbb{A}))\bigr)$-Hilbert correspondence. 
Let $\varepsilon_P\colon P(\mathbb{A})\longrightarrow
M(\mathbb{A})=P(\mathbb{A})/N(\mathbb{A})$ be the quotient map.
We define two Hilbert correspondences. 
\begin{enumerate}
    \item Consider $C^*(M(\mathbb{A}))$ as a $(C^*(P(\mathbb{A})),C^*(M(\mathbb{A})))$ correspondence with the left $C^*(P(\mathbb{A}))$-action given by $\varepsilon_P$. 
    We let
    \begin{equation}\label{eXPMG}
        X_{P,M}^{G}:=X_{P}^{G}
\otimes_{C^{*}(P(\mathbb{A}))}
C^{*}(M(\mathbb{A})). 
    \end{equation}
For every unitary representation $(\sigma,H_{\sigma})$ of $M(\mathbb{A})$,
we obtain
\begin{equation*}
\begin{aligned}
    X_{P,M}^{G}
\otimes_{C^{*}(M(\mathbb{A}))}
H_{\sigma}
&\cong
X_{P(\mathbb{A})}^{G(\mathbb{A})}
\otimes_{C^{*}(P(\mathbb{A}))}
C^{*}(M(\mathbb{A}))
\otimes_{C^{*}(M(\mathbb{A}))}
H_{\sigma}
\\
&\cong
X_{P(\mathbb{A})}^{G(\mathbb{A})}
\otimes_{C^{*}(P(\mathbb{A}))}
H_{\sigma\circ q}
\cong
\Ind_{P(\mathbb{A})}^{G(\mathbb{A})}\sigma.
\end{aligned}
\end{equation*}
This realizes the parabolic inductions from $P(\mathbb{A})$ to $G(\mathbb{A})$. 

\item
For the automorphic representation $R_G$ of $G(\mathbb{A})$ on $L^2(G(\mathbb{Q})\backslash G(\mathbb{A}))$ together with 
$R_G\colon C^{*}(G(\mathbb{A}))\to B(L^2(G(\mathbb{Q})\backslash G(\mathbb{A})))$, 
we let $I_{G}=\ker R_{G}$ and consider the quotient map
\begin{center}
    $\varepsilon_{G,\mathcal{A}}\colon C^{*}(G(\mathbb{A})) \to C^{*}_{\mathcal{A}}(G(\mathbb{A}))\cong C^*(G(\mathbb{A}))/I_{G}$. 
\end{center}
Consider the Levi subgroup $M$ and the quotient map
\begin{center}
    $\varepsilon_{M,\mathcal{A}}\colon C^{*}(M(\mathbb{A})) \to C^{*}_{\mathcal{A}}(M(\mathbb{A}))\cong C^*(M(\mathbb{A}))/I_{M}$. 
\end{center}
We know that $C^{*}_{\mathcal{A}}(M(\mathbb{A}))$ is a $(C^{*}(M(\mathbb{A})),C^{*}_{\mathcal{A}}(M(\mathbb{A})))$ correspondence with the left action of $C^{*}(M(\mathbb{A}))$ given by $\varepsilon_{M,\mathcal{A}}$. 
We let
\begin{equation}\label{eXPA}
    X_{P}^{\mathcal{A}}:=
X_{P,M}^{G}
\otimes_{C^*(M(\mathbb{A}))}
C^*_{\mathcal{A}}(M(\mathbb{A})). 
\end{equation}
It is known that $X_{P}^{\mathcal{A}}$ is a $(C^{*}(G(\mathbb{A})),C^{*}_{\mathcal{A}}(M(\mathbb{A})))$ correspondence, which is equipped with the morphism
\begin{center}
    $\Phi_{P}^{\mathcal{A}}:
C^{*}(G(\mathbb{A}))
\longrightarrow
\mathcal{L}_{C^{*}_{\mathcal{A}}(M(\mathbb{A}))}(X_{P}^{\mathcal{A}})$. 
\end{center}
\end{enumerate}

\subsection{The $(C^{*}_{\mathcal{A}}(G(\mathbb{A})),C^{*}_{\mathcal{A}}(M(\mathbb{A})))$-Hilbert correspondence}

\begin{proposition}\label{pGviaGA}
$X_{P}^{\mathcal{A}}$ is a $(C^{*}_{\mathcal{A}}(G(\mathbb{A})),C^{*}_{\mathcal{A}}(M(\mathbb{A})))$-correspondence, i.e.,
the map $\Phi_{P}^{\mathcal{A}}$ factors through $\varepsilon_{G,\mathcal{A}}$.    
\end{proposition}
\begin{proof}
Since $R_M$ gives a faithful 
representation of $C^{*}_{\mathcal{A}}(M(\mathbb{A}))$, the induced Hilbert-space representation
satisfies
\begin{center}$
X_{P}^{\mathcal{A}}
\otimes_{C^{*}_{\mathcal{A}}(M(\mathbb{A}))}
L^{2}(M(\mathbb{Q})\backslash M(\mathbb{A}))
\cong
\Ind_{P}^{G}R_M.
$\end{center}
It follows that
\begin{center}$
\ker\Phi_{P}^{\mathcal{A}}
=
\ker\left(
\Ind_{P(\mathbb{A})}^{G(\mathbb{A})}
R_M
\right).
$\end{center}
Therefore the left action factors through
$C_{\mathcal{A}}^{*}(G)$ if and only if
\begin{equation}\label{eRMRG}
  \Ind_{P(\mathbb{A})}^{G(\mathbb{A})}
R_M
\prec
R_G,  
\end{equation}
where $\pi \prec \rho$ means that $\pi$ is  weakly contained in $\rho$.  

Note that $\ker R_{M}=\cap_{\sigma\in \Pi_{\mathcal{A}}(M)}\ker \sigma$ and parabolic induction preserves weak containment. 
We may consider a single automorphic representation. 
\begin{enumerate}
    \item For an irreducible automorphic representation $\sigma$ appearing in the discrete spectrum of $R_M$, we have $\Ind_{P(\mathbb{A})}^{G(\mathbb{A})}\sigma\prec R_G$ by \eqref{eautodecP}. 

    \item If $\sigma$ lies in the continuous spectrum (the direct integral) of $R_M$, we apply \eqref{eautodecP} to $M$. 
By Theorem \ref{tLldspara}, 
there exists a parabolic subgroup $Q$ of $M$ with its Levi decomposition $Q=M_{Q}N_{Q}$ such that
\begin{center}
    $\sigma \prec \Ind_{Q(\mathbb{A})}^{M(\mathbb{A})}
\bigl(\delta_{\lambda}\bigr)$
\end{center}
for some cuspidal automorphic representation $\delta$ of $M_{Q}(\mathbb{A})$ and some unitary character $e^\lambda$ of $A_{M_{Q}}(\mathbb{R})^0$. 
Moreover, $QN_P$ is a parabolic subgroup of $G$ with Levi factor $M_Q$. 
By the transitivity of induced representations, we obtain
\begin{center}
$\Ind_{P(\mathbb{A})}^{G(\mathbb{A})}
\left(
\Ind_{Q(\mathbb{A})}^{M(\mathbb{A})}
\delta_{\lambda}
\right)
\cong
\Ind_{(QN_P)(\mathbb{A})}^{G(\mathbb{A})}
\delta_{\lambda}$, 
\end{center}
which is weakly contained in $R_G$. 
\end{enumerate}
Thus, for each $\sigma\prec R_M$, we have $\Ind_{P(\mathbb{A})}^{G(\mathbb{A})}\sigma\prec R_G$ which proves \eqref{eRMRG}. 
\end{proof}

\begin{corollary}\label{cPindtensor}
Let $(\sigma,H_{\sigma})$ be an automorphic representation of $M(\mathbb{A})$. 
We have
\begin{center}
$\Ind_{P(\mathbb{A})}^{G(\mathbb{A})}H_{\sigma}\cong X_{P}^{\mathcal{A}}\otimes_{C^*_{\mathcal{A}}(M(\mathbb{A}))}H_{\sigma}$
\end{center}
as automorphic representations of $G(\mathbb{A})$. 
\end{corollary}

\begin{example}\label{exmptransind}
Assume $G=\GL(5)$ and take the parabolic subgroup $P=M_{P}N_{P}$ as follows.
\begin{center}
    $P=\left( \begin{array}{ccccc}
* & * & * & * & * \\
* & * & * & * & * \\
* & * & * & * & * \\
0 & 0 & 0 & * & * \\
0 & 0 & 0 & * & * 
\end{array} \right),
M_P=\left( \begin{array}{ccccc}
* & * & * & 0 & 0 \\
* & * & * & 0 & 0 \\
* & * & * & 0 & 0 \\
0 & 0 & 0 & * & * \\
0 & 0 & 0 & * & * 
\end{array} \right)\cong \GL(3)\times \GL(2)$. 
\end{center}
We take the following parabolic subgroup $Q$ with its Levi decomposition $Q=M_Q N_Q$:
\begin{center}
    $Q=\left( \begin{array}{ccccc}
* & * & * & 0 & 0 \\
* & * & * & 0 & 0 \\
0 & 0 & * & 0 & 0 \\
0 & 0 & 0 & * & * \\
0 & 0 & 0 & * & * 
\end{array} \right),
M_Q=\left( \begin{array}{ccccc}
* & * & 0 & 0 & 0 \\
* & * & 0 & 0 & 0 \\
0 & 0 & * & 0 & 0 \\
0 & 0 & 0 & * & * \\
0 & 0 & 0 & * & * 
\end{array} \right)\cong \GL(2)\times \GL(1)\times \GL(2)$. 
\end{center}
Then $QN_P$ is the following parabolic subgroup
\begin{center}
    $QN_P=\left( \begin{array}{ccccc}
* & * & * & * & * \\
* & * & * & * & * \\
0 & 0 & * & * & * \\
0 & 0 & 0 & * & * \\
0 & 0 & 0 & * & * 
\end{array} \right)
$,
\end{center}
whose Levi factor is $M_Q$. 
\end{example}

We end this section with an explicit description of $X_{P}^{\mathcal{A}}$.

\begin{remark}\label{rXidX}
Recall that the decomposition $P=M_PN_P$ gives a quotient homomorphism
\begin{center}
$
\varepsilon_{P}\colon P(\mathbb{A})\longrightarrow M(\mathbb{A}),
\qquad
\varepsilon_{P}(mn)=m.
$
\end{center}
Let $\delta_{P}$ be the modulus character of $P(\mathbb{A})$, given by
\begin{center}
$
\delta_{P}(p)
=
\left|
\det\left(
\operatorname{Ad}(p)|_{\mathfrak{n}_{P}}
\right)
\right|_{\mathbb{A}},
$
\end{center}
where $\mathfrak{n}_{P}$ stands for the Lie algebra of $N_P$. 
We let $U_M$ denote the canonical map from $M(\mathbb{A})$ to the unitary group of the multiplier algebra of $C^{*}_{\mathcal{A}}(M(\mathbb{A}))$, which is given by the composition
\begin{center}
    $M(\mathbb{A})\to U(M(C^{*}(M(\mathbb{A}))))\to U(M(C^{*}_{\mathcal{A}}(M(\mathbb{A}))))$. 
\end{center}
Here the first map is the canonical universal unitary representation and the second one comes from the quotient map $C^{*}(M(\mathbb{A}))\to C^{*}_{\mathcal{A}}(M(\mathbb{A}))$. 
Define $X_{P}^{\mathcal{A},0}$ to be the space of continuous functions on $G(\mathbb{A})$ such that
\begin{itemize}
    \item $\supp{F}$ is compact modulo $P(\mathbb{A})$,
    \item $F(pg)=\delta_{P}(p)^{1/2}
U_{M}(\varepsilon_{P}(p))F(g)$.  
\end{itemize}
The right $C^{*}_{\mathcal{A}}(M(\mathbb{A}))$-module structure is defined by
\begin{center}
$
(F\cdot b)(g)=F(g)b 
$ 
\end{center}
for $b\in C^{*}_{\mathcal{A}}(M(\mathbb{A}))$. 
The group $G(\mathbb{A})$ acts on $X_{P}^{\mathcal{A},0}$ by right translation:
\begin{center}
$
\bigl(R(g_{0})F\bigr)(g)
=
F(gg_{0})$
\end{center}
for $g_0,g\in G(\mathbb{A})$. 
This action extends to an action of $C^{*}(G(\mathbb{A}))$ by adjointable operators on $X_{P}^{\mathcal{A}}$. 

Let $K\subset G(\mathbb{A})$ be a maximal compact subgroup such that
\begin{center}
$
G(\mathbb{A})=P(\mathbb{A})K,
$
\end{center}
and normalize the Haar measure on $K$ to have total mass one. For $F_{1},F_{2}\in X_{P}^{\mathcal{A},0}$, define
\begin{center}
$
\langle F_{1},F_{2}\rangle_{C_{\mathcal{A}}^{*}(M)}
=
\int_{K}
F_{1}(k)^{*}F_{2}(k),dk.
$
\end{center}
The completion of $X_{P}^{\mathcal{A},0}$ with respect to the norm
\begin{center}
$
\|F\|
=
\|
\langle F,F\rangle_{C_{\mathcal{A}}^{*}(M)}
\|^{1/2}
$
\end{center}
is canonically isomorphic to
\begin{center}
$
X_{P}^{\mathcal{A}}=
X_{P(\mathbb{A})}^{G(\mathbb{A})}
\otimes_{C^{*}(P(\mathbb{A}))}
C^{*}_{\mathcal{A}}(M(\mathbb{A})),
$
\end{center}
where $C^*(P(\mathbb{A}))$ acts on $C^{*}_{\mathcal{A}}(M(\mathbb{A}))$ through the map $C^*(P(\mathbb{A}))\to C^{*}(M(\mathbb{A}))\to C^{*}_{\mathcal{A}}(M(\mathbb{A}))$. 

\end{remark}


\section{The compact actions of $C^*_{\mathcal{A}}(G(\mathbb{A}))$}
\label{scptCAG}

Throughout this section, we abbreviate $G(\mathbb{A}),P(\mathbb{A}),M_{P}(\mathbb{A})$ and $M_{P}(\mathbb{A})^{1}$ by $G,P,M$ and $M^1$, respectively. 
Let $X_{P}^{\mathcal{A}}$ be the $(C^{*}_{\mathcal{A}}(G),C^{*}_{\mathcal{A}}(M))$-correspondence constructed in \eqref{eXPA} above. 
For a Hilbert correspondence  $\mathcal{H}$ from $C^*_{\mathcal{A}}(M)$ to a $C^*$-algebra $C$,  
we define the Hilbert module over $C$ to be
\begin{center}
    $\Ind_{P}^{G}\mathcal{H}:=\Ind_{P}^{G}(X_{P}^{\mathcal{A}},\mathcal{H})=X_{P}^{\mathcal{A}}\otimes_{C^*_{\mathcal{A}}(M)}\mathcal{H}$,
\end{center}
which is a $(C^*_{\mathcal{A}}(G(\mathbb{A})),C)$-Hilbert correspondence.


Consider the action of $G$ on the space $G/P$ by left translation. 
This gives a dynamical system $(G,C(G/P),\alpha)$ with $\alpha_g(f)(x)=f(g^{-1}x)$.  
The full crossed product $C(G/P)\rtimes G$
is the completion of the convolution $*$-algebra $C_c(G,C(G/P))$ 
with respect to the universal $C^*$-norm, where the product is defined by
\begin{center}
   $(F_1*F_2)(g)
=
\int_G F_1(h)\alpha_h(F_2(h^{-1}g))\,dh$, 
\end{center}
where $F_1,F_2\colon G\to C(G/P)$ belong to $C_c(G,C(G/P))$. 
Equivalently, it is the universal $C^*$-algebra for covariant representations
$(\pi,u)$ with $\pi\colon C(G/P)\to B(H)$ and $u\colon G\to U(H)$, which satisfy $u_g\pi(f)u_g^*=\pi(\alpha_g(f))$. 

Recall that the multiplier algebra $M(A)$ of a $C^*$-algebra $A$ can be identified with the unital $C^*$-algebra of multipliers of $A$; equivalently, $M(A)=\mathcal{L}_A(A)$ (see \cite[\S 2]{Lance1995}). 
The crossed product has a canonical strictly continuous unitary representation
\begin{center}
   $u:G\longrightarrow U(M(C(G/P)\rtimes G))$, 
\end{center}
which is called the canonical implementing representation. 
By the universal property
of the full group $C^*$-algebra, it induces a nondegenerate $*$-homomorphism
\begin{center}
  $C^*(G)\longrightarrow M(C(G/P)\rtimes G)$, 
\end{center}
which sends the canonical unitary operator of $C^*(G)$ associated with $g\in G$ to the multiplier $u_g$. 
On the dense subalgebra $C_c(G)\subset C^*(G)$, this map is given, for $f\in C_c(G)$, by
\begin{center}
    $f\longmapsto
\bigl(g\mapsto f(g)1_{C(G/P)}\bigr)$, 
\end{center}
which belongs to $ C_c(G,C(G/P))$ as well as $C(G/P)\rtimes G$. 
Hence the canonical copy of $C^*(G)$ acts by multipliers on the crossed
product through the unitary representation implementing the action of $G$ on $G/P$.

\begin{proposition}\label{pGAcptX}
The $C^*$-algebra $C^*_{\mathcal{A}}(G(\mathbb{A}))$ acts by compact operators on the Hilbert module $X_{P}^{\mathcal{A}}$ over $C^*_{\mathcal{A}}(M(\mathbb{A}))$. 
\end{proposition}
\begin{proof}
Recall that $X_{P}^{\mathcal{A}}\cong X_{P}^{G}\otimes_{C^{*}(P)}C^{*}(M)\otimes_{C^*(M)}C^*_{\mathcal{A}}(M)$. 
We first show that $C^*(G)$ acts on $X_{P}^{G}$ by $C^*(P)$-compact operators. 
By the Imprimitivity Theorem (see \cite[Section C.4]{DanaWilliams2007}), we know
\begin{center}
$\mathcal{K}_{C^{*}(P)}(X_{P}^{G})\cong
C_{0}(G/P)\rtimes G$, 
\end{center}
where the right side is the full crossed product.
Choose a maximal compact subgroup $K\subset G$. 
The adelic Iwasawa decomposition gives $G=PK$. 
Consequently, we have
\begin{center}
$G/P\cong K/(K\cap P)$, 
\end{center}
and hence $G/P$ is compact. 
Therefore, $C_{0}(G/P)=C(G/P)$, which is a unital $C^*$-algebra. 

Recall that the canonical unitary representation of $G$ in the multiplier algebra of the crossed product defines a homomorphism
\begin{center}
$C^{*}(G)\longrightarrow M\bigl(C(G/P)\rtimes G\bigr)$. 
\end{center}
Since $C(G/P)$ is unital, the integrated canonical representation sends $C_{c}(G)$ into $C_{c}\bigl(G,C(G/P)\bigr)$. 
Thus, by continuity and density, it sends  $C^*(G)$ into $C(G/P)\rtimes G$. 
More precisely, if $f\in C_{c}(G)$, then the corresponding element of the crossed product is represented by the compactly supported function
\begin{center}
$g\longmapsto f(g)1_{C(G/P)}$, 
\end{center}
which belongs to $C_{c}\bigl(G,C(G/P)\bigr)$. 
It follows by continuity and density of $C_{c}(G)$ in $C^{*}(G)$ that
\begin{center}
$
C^{*}(G)
\longrightarrow
C(G/P)\rtimes G
\cong
\mathcal{K}_{C^*(P)}(X_P^G).
$
\end{center}
Thus the left action $C^{*}(G)
\longrightarrow
\mathcal{L}_{C^*(P)}(X_P^G)$ has its image in the compact operators $\mathcal{K}_{C^*(P)}(X_P^G)$. 

Consider the tensor product $X_{P}^{G}\otimes_{C^{*}(P)}C^{*}(M)\otimes_{C^*(M)}C^*_{\mathcal{A}}(M)$. 
We have an action of $C^*(P)$ by compact operators on $C^{*}(M)$ since the action factors through $C^{*}(M)$. 
Thus, by Lemma \ref{lcpttensor}, $C^*(G)$ acts by $C^{*}(M)$-compact operators on $X_{P}^{G}\otimes_{C^{*}(P)}C^{*}(M)$. 
Since $C^*_{\mathcal{A}}(M)$ is a $C^*$-quotient algebra of $C^*(M)$, we have
\begin{center}
    $\mathcal{K}_{C^*_{\mathcal{A}}(M)}(C^*_{\mathcal{A}}(M))\cong C^*_{\mathcal{A}}(M)$
\end{center}
which implies that $C^*(M)$ acts on it by compact operators via the map $C^*(M)\to C^*_{\mathcal{A}}(M)$. 
By Lemma \ref{lcpttensor} again, $C^*(G)$ acts by $C^*_{\mathcal{A}}(M)$-compact operators on $X_{P}^{\mathcal{A}}$. 
The result then follows from Proposition \ref{pGviaGA}.  
\end{proof}

\begin{corollary}
If $C^*_{\mathcal{A}}(M(\mathbb{A}))$ acts through compact operators on a Hilbert module $\mathcal{H}$ over it, then $C^*_{\mathcal{A}}(G(\mathbb{A}))$ acts through compact operators on $\Ind_{P}^{G}\mathcal{H}$. 
\end{corollary}
\begin{proof}
Since $\Ind_{P}^{G}\mathcal{H}=X_{P}^{\mathcal{A}}\otimes_{C^*_{\mathcal{A}}(M)}\mathcal{H}$, it follows from Proposition \ref{pGAcptX} and Lemma \ref{lcpttensor}. 
\end{proof}

\begin{lemma}\label{ladeleGCCR}
The $C^*$-algebra of an adelic reductive group $G(\mathbb{A})$ acts through compact operators on  any irreducible unitary representation of $G(\mathbb{A})$.   
\end{lemma}
\begin{proof}
Note that $G(\mathbb{A})=\prod_v' G(F_v)$
with respect to hyperspecial compact open subgroups $K_v\subset G(F_v)$ for almost all finite places $v$. 
We also know that $G(F_v)$ is a CCR group for each $v$, i.e. $\rho(C^*(G(F_v)))=\mathcal{K}(H_{\rho})$ for any irreducible unitary representation $(\rho,H_{\rho})$ of $G(F_v)$ (see \cite{HC1953I} for archimedean fields and \cite{Bernstein1974,BekkaEcht2020} for nonarchimedean fields of characteristic $0$). 
Let  
$(\pi,H)$ be an irreducible unitary  representation of  $G(\mathbb{A})$. 
By \cite{Flath77}, we have the following restricted tensor product decompositions: 
\begin{center}
    $\pi\simeq\bigotimes_v'\pi_v,
\qquad
H_\pi\simeq\bigotimes_v'(H_{\pi_v},\xi_v)$,
\end{center}
where each $(\pi_v,H_{\pi_v})$ is an irreducible unitary  representation of  $G(F_v)$ and $\xi_v\in H_{\pi_v}^{K_v}$ such that $\dim H_{\pi_v}^{K_v}=1$ for almost all finite places $v$.

The full group $C^*$-algebra also admits the corresponding restricted tensor
product decomposition
\begin{center}
   $C^*(G(\mathbb{A}))
\simeq
\bigotimes_v'
\bigl(C^*(G(F_v)),p_{K_v}\bigr)$,  
\end{center}
where $p_{K_v}\in C^*(G(F_v))$ is the projection associated with the
normalized characteristic function of $K_v$ (see \cite[Corollary 3.8]{GMS2026}). 
On the other hand, we have $\pi_v(p_{K_v})=P_{H_{\pi_v}^{K_v}}$. 
Hence, for almost all $v$, $\pi_v(p_{K_v})=P_{\xi_v}$, which is the rank-one projection onto $\mathbb C\xi_v$.
Therefore,
\begin{equation*}
\begin{aligned}
\pi\bigl(C^*(G(\mathbb{A}))\bigr)
&\simeq
\bigotimes_v'
\left(
\pi_v(C^*(G(F_v))),\pi_v(p_{K_v})
\right)\\
&=
\bigotimes_v'
\left(
\mathcal K(H_{\pi_v}),P_{\xi_v}
\right).
\end{aligned}    
\end{equation*}
By \cite[Proposition 4.7]{GMS2026}, we know
\begin{center}
  $\bigotimes_v'
\left(
\mathcal K(H_{\pi_v}),P_{\xi_v}
\right)
\simeq
\mathcal K\left(
\bigotimes_v'(H_{\pi_v},\xi_v)
\right)$. 
\end{center}
This is to say
\begin{center}
  $\pi\bigl(C^*(G(\mathbb{A}))\bigr)=\mathcal K(H_\pi)$, 
\end{center}
which proves that $C^*(G(\mathbb{A}))$ is CCR.
\end{proof}

We let $A=A_{M_P}(\mathbb{R})^0$. 
For an automorphic representation $(\sigma,H_{\sigma})$ of $M^1(\mathbb{A})$, we define
\begin{center}
    $\mathcal{H}_{\sigma}:=C_{0}(\widehat{A},H_{\sigma})$. 
\end{center}
This is the Hilbert $C_{0}(\widehat{A})$-module of $H_{\sigma}$-valued continuous functions on $\widehat{A}$ which vanish at infinity. 
Note that every unitary character of $A$ is uniquely of the form $\chi_\lambda(a)=e^{\langle \lambda,\log a\rangle}$ with $\lambda\in i\mathfrak{a}_P^*$. 
We may identify $\widehat{A}$  with $i\mathfrak{a}_P^*$. 
We may also identify 
$\widehat{A}$ with $i\mathfrak{a}_P$ by choosing an inner product identifying
$\mathfrak{a}_P^*\simeq\mathfrak{a}_P$. 

\begin{lemma}\label{lCAMCAM1}
$C^*_{\mathcal{A}}(M)\cong C^*_{\mathcal{A}}(M^1)\otimes C_0(\widehat{A})$.
\end{lemma}
\begin{proof}
 As $M(\mathbb{A})\simeq M^{1}(\mathbb{A})\times A$, we have
\begin{equation}
    L^{2}\bigl(M(\mathbb{Q})\backslash M(\mathbb{A})\bigr)
\simeq
L^{2}\bigl(M(\mathbb{Q})\backslash M^{1}(\mathbb{A})\bigr)
\widehat{\otimes}L^{2}(A). 
\end{equation}
Since $L^{2}(A)\simeq\int_{\widehat A}^{\oplus}\mathbb{C}_{\chi}\,d\chi$, we further obtain
\begin{equation}\label{eRMRM1}
    R_M\simeq
\int_{\widehat A}^{\oplus}
\left(R_{M^{1}}\otimes\chi\right)\,d\chi.
\end{equation}
More explicitly, write 
\begin{center}
    $R_{M^{1}}
\simeq
\int_{\Pi_{\mathcal{A}}(M^{1})}^{\oplus}
\sigma\otimes 1_{\mathcal{M}_{\sigma}}\,d\mu(\sigma)$,
\end{center}
where $1_{\mathcal{M}_{\sigma}}$ denotes the identity action on the multiplicity space $\mathcal{M}_{\sigma}$ of $\sigma$. 
Then we can rewrite \eqref{eRMRM1} as
\begin{equation}
  R_M\simeq
\int_{\Pi_{\mathcal{A}}(M^{1})}^{\oplus}
\int_{\widehat A}^{\oplus}
(\sigma\otimes\chi)\otimes 1_{\mathcal{M}_{\sigma}}\,
d\chi\,d\mu(\sigma).   
\end{equation}
The result then follows from the facts that $C^*(M(\mathbb{A}))\cong C^*(M(\mathbb{A})^1)\otimes C^*(A)$ and $C^*(A)\cong C_0(\widehat{A})$.
\end{proof}

We take $f=f_1\otimes \chi\in C^*_{\mathcal{A}}(M)$ with $f_1\in  C^*_{\mathcal{A}}(M^1)$ and $\chi\in C_0(\widehat{A})$. 
Then the action of $C^*_{\mathcal{A}}(M)$ on $\mathcal{H}_{\sigma}$ is given explicitly as
\begin{equation}\label{eCAMHsigma}
    (f\cdot h)(\varphi)=(\sigma\otimes\varphi)(f)h(\varphi)=(\sigma(f_1)\otimes\chi(\varphi))h(\varphi) 
\end{equation}
for $h\in \mathcal{H}_{\sigma}$ and $\varphi\in \widehat{A}$ by Lemma \ref{lCAMCAM1}.

\begin{lemma}\label{lCAMcpt}
$C^*_{\mathcal{A}}(M(\mathbb{A}))$ acts by compact operators on the Hilbert module $\mathcal{H}_{\sigma}$. 
\end{lemma}
\begin{proof}
By Lemma \ref{lCAMCAM1}, it suffices to show that $C^*_{\mathcal{A}}(M^1)\otimes C_0(\widehat{A})$ acts as compact operators on $\mathcal{H}_{\sigma}$. 
As $M(\mathbb{A})\cong M(\mathbb{A})^1\times A_{M}(\mathbb{R})^0$, an irreducible representation $\rho$ of $M(\mathbb{A})^1$ also gives an irreducible one of $M(\mathbb{A})$ by letting $A_{M}(\mathbb{R})^0$ act trivially. 
By Lemma \ref{ladeleGCCR}, $M(\mathbb{A})^1$ is also CCR. 
The first factor of $C^*_{\mathcal{A}}(M^1)$ acts through compact operators on $H_{\sigma}$. 
So $C^*_{\mathcal{A}}(M(\mathbb{A}))$ acts through the $C^*$-algebra $C_0(\widehat{A},\mathcal{K}(H_{\sigma}))$, which can be identified with $\mathcal{K}_{C_0(\widehat{A})}(\mathcal{H}_{\sigma})$.
\end{proof}

Recall that $\Ind_P^G\mathcal{H}_{\sigma}$ is a Hilbert $(C^*_{\mathcal{A}}(G),C_0(\widehat{A}))$-correspondence. 

\begin{corollary}\label{cCAGcptHsigma}
$C^*_{\mathcal{A}}(G(\mathbb{A}))$ acts by compact operators on the Hilbert module $\Ind_{P}^{G}\mathcal{H}_{\sigma}$.   
\end{corollary}
\begin{proof}
It follows from Lemma \ref{lCAMcpt}, Proposition \ref{pGAcptX} and Lemma \ref{lcpttensor}. 
\end{proof}

\section{The Langlands spectral decomposition of $C^*_{\mathcal{A}}(G(\mathbb{A}))$}\label{sinjG}

By Corollary \ref{cCAGcptHsigma}, $C^*_{\mathcal{A}}(G(\mathbb{A}))$ acts as $C_0(\widehat{A})$-compact operators on the Hilbert $C_0(\widehat{A})$-module $\Ind_{P}^{G}\mathcal{H}_{\sigma}$ as shown in Section \ref{scptCAG}. 
Recall the full discrete spectrum: 
    \begin{center}
        $H_{M,\disc}=L^2_{\disc}
\bigl(M(\mathbb{Q})\backslash M(\mathbb{A})^1\bigr)\cong 
\bigoplus_{\sigma\in\Pi_{\disc}(M^1)}m(\sigma)\cdot H_\sigma$ 
    \end{center}
with each $m(\sigma)$ finite.
We also recall the $\sigma$-isotypic subspace of $H_{M,\disc}$: 
\begin{center}
   $H_{M,\sigma}=L^2_{\disc}
\bigl(M(\mathbb{Q})\backslash M(\mathbb{A})^1\bigr)[\sigma]=\oplus_{k=1}^{m(\sigma)}H_\sigma$. 
\end{center}
We can form the Hilbert $C_0(\widehat{A})$-modules
\begin{center}
$\mathcal{H}_{M,\disc}:=C_{0}(\widehat{A},H_{M,\disc})$, and $\mathcal{H}_{M,\sigma}:=C_{0}(\widehat{A},H_{M,\sigma})$ 
\end{center}
as well as the Hilbert $(C^*_{\mathcal{A}}(G),C_0(\widehat{A}))$-correspondences
\begin{center}
$\Ind_P^G \mathcal{H}_{M,\disc}=\Ind_{P}^{G}\bigl(C_{0}(\widehat{A},H_{M,\disc})\bigr)$, and $\Ind_P^G \mathcal{H}_{M,\sigma}=\Ind_{P}^{G}\bigl(C_{0}(\widehat{A},H_{M,\sigma})\bigr)$.  
\end{center}

\subsection{The decomposition via discrete Levi datum}

\begin{lemma}\label{lCAGdiscCPT}
$C^*_{\mathcal A}(G(\mathbb{A}))$ acts as $C_{0}(\widehat{A})$-compact operators on the Hilbert $C_{0}(\widehat{A})$-modules $\Ind_{P}^{G}\mathcal{H}_{M,\disc}$ and $\Ind_P^G \mathcal{H}_{M,\sigma}$. 
\end{lemma}
\begin{proof}
As $m(\sigma)$ is finite, the compactness of the action on $\Ind_P^G \mathcal{H}_{M,\sigma}$ follows from Corollary  \ref{cCAGcptHsigma}. 

For $\Ind_{P}^{G}\mathcal{H}_{M,\disc}$, we take $f\in C^{\infty}_{c}(M(\mathbb{A})^1)$. 
It is known that $R_{M,\disc}(f)$ is of the trace class and thus compact (see \cite[Corollary 8.10]{Muller1989}). 
Note that $C^{\infty}_{c}(M(\mathbb{A})^1)$ is dense in $C^*_{\mathcal A}(M(\mathbb{A})^1)$ and the compact operators are norm closed.  
Thus $C^*_{\mathcal A}(M(\mathbb{A})^1)$ acts as compact operators on $H_{M,\disc}$. 
Hence $C^*_{\mathcal A}(M(\mathbb{A}))$ acts as compact operators on $\mathcal{H}_{M,\disc}$ by Lemma \ref{lCAMCAM1} (see also Lemma \ref{lCAMcpt}). 
Hence, by Proposition \ref{pGAcptX} and Lemma \ref{lcpttensor}, $C^*_{\mathcal A}(G(\mathbb{A}))$ acts by compact operators on $\Ind_{P}^{G}\mathcal{H}_{M,\disc}$. 
\end{proof}

\begin{lemma}\label{lHmod0}
Let $E=\bigoplus_{i \in I} E_i$ be a Hilbert $B$-module and let $T=\bigoplus_{i \in I}T_i$ be a diagonal operator on $E$. 
If $T\in\mathcal K_B(E)$, then
\begin{center}
    $\#\{i\in I\mid \|T_i\|\geq \varepsilon\}<\infty$ 
\end{center}
for every $\varepsilon>0$. 
\end{lemma}

\begin{proof}
For a finite subset $J\subset I$, 
let $p_J$ denote the projection onto the finite direct sum $\bigoplus_{i\in J}E_i$. 
For a rank-one operator $\theta_{\xi,\eta}$, 
we have $\|(1-p_J)\theta_{\xi,\eta}\|\longrightarrow0$ 
for sufficiently large $J$. 
Hence, by norm density of finite rank operators in
$\mathcal K_B(E)$, we know $\|(1-p_J)T\|
\longrightarrow0$. 
If $i\notin J$, then $\|T_i\|=\|p_iTp_i\|\leq\|(1-p_J)T\|$, which completes the proof. 
\end{proof}

Now we are able to define the $*$-homomorphism 
\begin{center}
 $\Phi_{M,\sigma}\colon C^*_{\mathcal{A}}(G(\mathbb{A}))\to\mathcal{K}_{C_0(\widehat{A})}\bigl(\Ind_P^G \mathcal{H}_{M,\sigma}\bigr)$  
\end{center}
whose image will be denoted by $C^*_{\mathcal A}(G(\mathbb{A}))_{M,\sigma}$. 
\begin{proposition}\label{pweldefhom}
There is a well-defined homomorphism
\begin{equation}\label{ePsiCAG}
    \Phi:=\bigoplus_{[M,\sigma]}\Phi_{M,\sigma}\colon C^{*}_{\mathcal{A}}(G(\mathbb{A}))
\longrightarrow
\bigoplus_{[M,\sigma]}
C^*_{\mathcal A}(G(\mathbb{A}))_{M,\sigma}
\end{equation}
such that $\lim_{[M,\sigma]\rightarrow\infty}
\|\Phi_{M,\sigma}(a)\|=0 $ for every $a\in C^*_{\mathcal{A}}(G(\mathbb{A}))$.  
\end{proposition}
\begin{proof}
Take $f\in C^*_{\mathcal{A}}(G(\mathbb{A}))$. 
By Lemma \ref{lCAGdiscCPT} and Lemma \ref{lHmod0}, we have $\|
\Phi_{M,\sigma}(f)
\|
\longrightarrow 0$ when $\sigma$ is outside a sufficiently large finite set $J\subset \Pi_{\disc}(M^1)$.   
Therefore the family $\left(
\Phi_{M,\sigma}(f)
\right)_{\sigma\in\Pi_{\disc}(M^1)}$
belongs to the $C^*$-algebraic direct sum $\bigoplus_{\sigma\in\Pi_{\disc}(M^1)}C^*_{\mathcal A}(G(\mathbb{A}))_{M,\sigma}
$. 
The conclusion then follows from the fact that there are only finitely many equivalence classes of Levi subgroups. 
\end{proof}

\begin{proposition}\label{pinjPhi}
The homomorphism $\Phi$ above 
is injective. 
\end{proposition}
\begin{proof}
By \eqref{edecMdisc}, we have the decomposition $R_G=\bigoplus_{[M,\sigma]}R_{M,\sigma}$. 
Suppose that $\Phi(f)=0$ for $f\in C^*_{\mathcal A}(G(\mathbb{A}))$. 
Then, by definition, $\Phi_{M,\sigma}(f)=0$ for every representative of an associate class of pairs $(M,\sigma)$.
Hence each $R_{M,\sigma}(f)$ vanishes and $R_G(f)=\widehat{\bigoplus}_{[M,\sigma]}
R_{M,\sigma}(f)=0$. 
Since $C^*_{\mathcal A}(G(\mathbb{A}))=R_G(C^*(G(\mathbb{A})))$, 
the representation $R_G$ is faithful on
$C^*_{\mathcal A}(G(\mathbb{A}))$. 
So $f=0$ and $\Phi$ is injective.
\end{proof}

We let $C^*_{\mathcal{A}}(G(\mathbb{A}))_{M,\sigma}\subset \mathcal{K}_{C_0(\widehat{A})}\bigl(\Ind_P^G \mathcal{H}_{M,\sigma}\bigr)$ 
be the image of $\Phi_{M,\sigma}$. 
We may view the representations of $C^*_{\mathcal{A}}(G(\mathbb{A}))_{M,\sigma}$ as the representations of $C^*_{\mathcal A}(G(\mathbb{A}))$ which factor through $\Phi_{M,\sigma}$. 

The following result is motivated by the real tempered case \cite[Lemma 5.10]{CCH2016}. 
\begin{lemma}\label{lCAGfactor}
The irreducible representations of $C^*_{\mathcal{A}}(G(\mathbb{A}))_{M,\sigma}$ are precisely the irreducible representations of $C^*_{\mathcal A}(G(\mathbb{A}))$, which are the irreducible constituents of $\Ind_{P}^{G}(\sigma\otimes\lambda)$ as $\lambda$ ranges over $\widehat{A}$. 
\end{lemma}
\begin{proof}
By \cite[Proposition 2.10.2]{DiCalg}, for a $C^*$-subalgebra $B_0$ of a $C^*$-algebra $B$, the irreducible representations of $B_0$ are precisely the restrictions of irreducible representations of $B$ to a $B_0$-invariant, $B_0$-irreducible subspace. 

Now we take $B_0=C^*_{\mathcal{A}}(G(\mathbb{A}))_{M,\sigma}$ and $B=\mathcal{K}_{C_0(\widehat{A})}\bigl(\Ind_P^G \mathcal{H}_{M,\sigma}\bigr)$. 
Note that the irreducible representations of $B$ are given by evaluation at $\lambda\in \widehat{A}$ on the Hilbert space $\Ind_P^G H_{\sigma\otimes\lambda}$. 
Let $\mathbb{C}_{\lambda}$ denote the space for the one-dimensional representation of $A$ given by $\lambda$. 
We have $\Ind_{P}^{G}\mathcal{H}_{M,\sigma}\otimes_{C_{0}(\widehat{A})}\mathbb{C}_{\lambda}\cong\Ind_{P}^{G}H_{\sigma\otimes\lambda}$, whose restriction to $B_0$ is exactly the irreducible constituents of $\Ind_{P}^{G}(\sigma\otimes\lambda)$. 
\end{proof}

\subsection{Intertwining operators and the fixed-point subalgebras}
\label{sssubGlobalWeylFixedTarget}

We take a standard parabolic subgroup $P=M_PN_P=MN$ of $G$.
Let $A=A_M(\mathbb R)^0$ and identify $\widehat A$ with $i\mathfrak{a}_M^*$ as above. 
Take $\sigma\in\Pi_{\disc}\bigl(M(\mathbb A)^1\bigr)$ and consider the $\sigma$-isotypic subspace $H_{M,\sigma}=L^2_{\disc}
\bigl(M(\mathbb{Q})\backslash M(\mathbb A)^1\bigr)[\sigma]$ of the full discrete automorphic spectrum. 
Recall that $\mathcal H_{M,\sigma}=C_0(\widehat{A},H_{M,\sigma})$. 
We use the following notation and facts. 
\begin{itemize}
\item Recall that $I_{P}(\lambda,-)=$ the induced representation $\Ind_{P(\mathbb A)}^{G(\mathbb A)}(R_{M,\disc,\lambda}\otimes 1_{N(\mathbb{A})})$ of $G(\mathbb A)$, whose underlying space, denoted by $H_P$, is independent of $\lambda$.

\item $R_{M,\disc,\sigma,\lambda}=$ the representation of $M^1(\mathbb{A})$ on $H_{M,\sigma}$;

\item $I_{P,\sigma}(\lambda,-)=$ the induced representation $\Ind_{P(\mathbb A)}^{G(\mathbb A)}(R_{M,\disc,\sigma,\lambda}\otimes 1_{N(\mathbb{A})})$ of $G(\mathbb A)$;

\item $H_{P,\sigma}=$ the underlying space of $\Ind_{P(\mathbb A)}^{G(\mathbb A)}(R_{M,\disc,\sigma,\lambda}\otimes 1_{N(\mathbb{A})})$, which is independent of $\lambda$;

\item $\mathcal{E}_{M,\sigma}:=
\Ind_{P(\mathbb A)}^{G(\mathbb A)}
\mathcal H_{M,\sigma}=\Ind_{P(\mathbb A)}^{G(\mathbb A)}\bigl(C_{0}(\widehat{A},H_{M,\sigma})\bigr)$;

\item $W(G,M,\sigma):=
\{w\in W(G,M):{}^w\sigma\simeq\sigma\}$. 
\end{itemize}
Note that
\begin{equation}
    \mathcal {E}_{M,\sigma}\cong C_{0}(\widehat{A},\Ind_{P(\mathbb A)}^{G(\mathbb A)}H_{M,\sigma})
\end{equation}
as a Hilbert $(C^*_{\mathcal{A}}(G),C_0(\widehat{A}))$-correspondence. 
The fiber at each $\lambda\in \widehat{A}$ carries the induced representation $\Ind_{P(\mathbb A)}^{G(\mathbb A)}(R_{M,\disc,\sigma}\otimes e^{\lambda})$ while the underlying space is independent of $\lambda$.  

Following \cite[\S 7]{Art05}, 
we shall use the global intertwining operators
\begin{center}
    $M(w,\lambda):H_P\longrightarrow H_P,$
\end{center}
for $w\in W(G,M_P)$ (see Lemma \ref{lM}). 
This is the specialization of the proceding construction to $P'=P$ within the notation there. 
By Lemma \ref{lM}, they satisfy
\begin{equation}
\label{eGlobalMintertwine}
M(w,\lambda)I_P(\lambda,g)
=
I_P(w\lambda,g)M(w,\lambda)
\end{equation}
for $g\in G(\mathbb A)$, and
\begin{equation}
\label{eGlobalMfunctional}
M(w_1w_2,\lambda)
=
M(w_1,w_2\lambda)M(w_2,\lambda).
\end{equation}
Moreover, for $\lambda\in i\mathfrak{a}_M^*$, the operator
$M(w,\lambda)$ extends to a unitary operator in $U(H_P)$.

For $s\in W(\mathfrak{a}_P,\mathfrak{a}_{P})$ and the corresponding $w\in W(G,M_P)$, we know that $M(w,\lambda)$ intertwines the representations $I_{P,\sigma}(\lambda,-)$ and $I_{P,{}^w\sigma}(w\lambda,-)$ (see \cite[\S 15, page 86]{Art05}). 
Thus, for $w\in W(G,M,\sigma)$ and $\lambda\in i\mathfrak{a}_M^*$, we have a unitary operator $M(w,\lambda)\colon H_{P,\sigma}\to H_{P,\sigma}$
such that
\begin{equation}\label{eMsigma}
    M(w,\lambda)I_{P,\sigma}(\lambda,g)=I_{P,\sigma}(w\lambda,g) M(w,\lambda)
\end{equation}
for $g\in G(\mathbb{A})$. 

For $w\in W(G,M)$, we have an automorphism $\alpha_w:C_0(\widehat{A})\longrightarrow C_0(\widehat{A})$ given by 
\begin{center}
$\alpha_w(f)(\lambda)=f(w^{-1}\lambda)$. 
\end{center}
We show that the normalized intertwining operators $U_w$ (see \eqref{enormintwop}) give a well-defined action of $W(G,M,\sigma)$. 
\begin{proposition}
\label{pGlobalWeylAction}
For every $w\in W(G,M,\sigma)$, 
we have an
$\alpha_w$-twisted unitary automorphism $U_w\colon \mathcal{E}_{M,\sigma}\to
\mathcal{E}_{M,\sigma}$ given by
\begin{equation}\label{eUw}
    (U_wF)(\lambda)=M(w,w^{-1}\lambda)F(w^{-1}\lambda)
\end{equation}
for $F\in \mathcal{E}_{M,\sigma}$, 
which commutes with the $G(\mathbb A)$-action on $\mathcal E_{M,\sigma}$. 
This gives an action $\beta$ of $W(G,M,\sigma)$ on $\mathcal{K}_{C_0(\widehat{A})}(\mathcal E_{M,\sigma})$ given by $\beta_w(T)=U_{w}TU_{w}^{-1}$.
\end{proposition}
\begin{proof}
Given $F\in C_0(\widehat{A},H_{P,\sigma})$,  
we have $M(w,w^{-1}\lambda)F(w^{-1}\lambda)\in H_{P,\sigma}$ by \eqref{eMsigma}. 
As $w\colon \widehat{A}\to \widehat{A}$ is a homeomorphism, the right-hand side vanishes at
infinity and $U_wF\in C_0(\widehat{A},H_{P,\sigma})$. 

\begin{enumerate}
\item The map $\lambda\longmapsto
M(w,\lambda)|_{H_{P,\sigma}}$ is strongly continuous. 
We take $\xi\in H_P$ and $\xi_{k}\in H_P^0$ (the $K$-finite vectors) such that $\|\xi-\xi_{k}\|\to 0$ as $H_P^0$ is dense in $H_P$. 
For $w\in W(G,M,\sigma)$, we know that $M(w,\lambda)$  
is analytic, hence continuous on $i\mathfrak{a}_M^*$ (see \cite[Theorem 7.2(a)]{Art05}). 
Let $\lambda_n\to\lambda$. By the unitarity of $M(w,\lambda)$, we have 
\begin{equation*}
\begin{aligned}
&\|M(w,\lambda_n)\xi-M(w,\lambda)\xi\|\\
\leq &
\|M(w,\lambda_n)(\xi-\xi_k)\|
+
\|M(w,\lambda_n)\xi_k-M(w,\lambda)\xi_k\|
+
\|M(w,\lambda)(\xi_k-\xi)\|
\\
\leq &
2\|\xi-\xi_k\|
+
\|M(w,\lambda_n)\xi_k-M(w,\lambda)\xi_k\|\to 0.
\end{aligned}
\end{equation*}
This proves the strong continuity.   
In particular, the restriction $\lambda\longmapsto
M(w,\lambda)|_{H_{P,\sigma}}$ is strongly continuous.

\item We show that $U_wF$ is continuous.  If $\lambda_n\to\lambda$, then we have
\begin{equation*}
 \begin{aligned}
&\|(U_wF)(\lambda_n)-(U_wF)(\lambda)\|\\
\leq &
\|F(w^{-1}\lambda_n)-F(w^{-1}\lambda)\|
+
\left\|
\left(
M(w,w^{-1}\lambda_n)
-
M(w,w^{-1}\lambda)
\right)
F(w^{-1}\lambda)
\right\|.
\end{aligned}   
\end{equation*}
The first term tends to zero by continuity of $F$, and the second tends
to zero by strong continuity of the intertwining operators.  
Hence $U_wF$ is continuous.

\item $U_w$ is an $\alpha$-twisted unitary automorphism in the sense of \eqref{etwistcover}. 
For $f\in C_0(\widehat{A})$, we have
\begin{equation*}
 \begin{aligned}
(U_w(Ff))(\lambda)
=
M(w,w^{-1}\lambda)
F(w^{-1}\lambda)
f(w^{-1}\lambda)
=
(U_wF)(\lambda)\alpha_w(f)(\lambda),
\end{aligned}   
\end{equation*}
and hence $U_w(Ff)=U_w(F)\alpha_w(f)$. 
Similarly, the unitarity of $M(w,w^{-1}\lambda)$ gives
\begin{center}
    $\langle U_wF_1,U_wF_2\rangle(\lambda)
=
\left\langle
F_1(w^{-1}\lambda),
F_2(w^{-1}\lambda)
\right\rangle=
\alpha_w
\bigl(
\langle F_1,F_2\rangle
\bigr)(\lambda)$. 
\end{center}

\item $w\mapsto U_w$ gives a group homomorphism.   
Let $w_1,w_2\in W(G,M,\sigma)$.  By definition, we have
\begin{equation*}
\begin{aligned}
(U_{w_1}U_{w_2}F)(\lambda)
&=
M(w_1,w_1^{-1}\lambda)
(U_{w_2}F)(w_1^{-1}\lambda)
\\
&=
M(w_1,w_1^{-1}\lambda)
M(w_2,w_2^{-1}w_1^{-1}\lambda)
F(w_2^{-1}w_1^{-1}\lambda).
\end{aligned}   
\end{equation*}
Set $\mu=(w_1w_2)^{-1}\lambda=
w_2^{-1}w_1^{-1}\lambda$. We have $w_2\mu=w_1^{-1}\lambda$. 
By the functional equation in Lemma \ref{lM}, we have
\begin{equation*}
 \begin{aligned}
(U_{w_1}U_{w_2}F)(\lambda)
=
M\bigl(
w_1w_2,(w_1w_2)^{-1}\lambda
\bigr)
F\bigl(
(w_1w_2)^{-1}\lambda
\bigr)=
(U_{w_1w_2}F)(\lambda).
\end{aligned}   
\end{equation*}
Thus $w\mapsto U_w$ gives a homomorphism.  Since $M(1,\lambda)=1$, we have
$U_1=1$ and $U_w^{-1}=U_{w^{-1}}$. 
Consequently, $w\mapsto U_w$ defines an action of $W(G,M,\sigma)$.

\item $U_w$ commutes with the $G(\mathbb A)$-action.
Given $F\in C_0(\widehat{A},H_{P,\sigma})$, we have $(\Pi(g)F)(\lambda)=
I_P(\lambda,g)F(\lambda)$ on each fiber. 
By Lemma \ref{lM}, we know $M(w,w^{-1}\lambda)
I_P(w^{-1}\lambda,g)=I_P(\lambda,g)
M(w,w^{-1}\lambda)$ and therefore
\begin{equation*}
    \begin{aligned}
(U_w(gF))(\lambda)
&=
M(w,w^{-1}\lambda)
I_P(w^{-1}\lambda,g)
F(w^{-1}\lambda)
\\
&=
I_P(\lambda,g)
M(w,w^{-1}\lambda)
F(w^{-1}\lambda)
=
(g(U_{w}F))(\lambda).
\end{aligned}
\end{equation*}

\item $\beta_w\in \Aut
\left(
\mathcal K_{C_0(\widehat{A})}(\mathcal E_{M,\sigma})
\right)$. 
We may start with a rank-one operator $\theta_{\xi,\eta}(\zeta)=
\xi\langle\eta,\zeta\rangle$,  
where we have $U_w\theta_{\xi,\eta}U_w^{-1}=
\theta_{U_w\xi,U_w\eta}$. 
Hence $\beta_w$ preserves the compact operators. 
\end{enumerate}
Thus we obtain a $G(\mathbb{A})$-intertwining action $\beta$ of $W(G,M,\sigma)$ on $\mathcal{K}_{C_0(\widehat{A})}(\mathcal E_{M,\sigma})$.
\end{proof}

\begin{corollary}
\label{cintocpminv}
$C^{*}_{\mathcal A}(G(\mathbb A))_{M,\sigma}
\subset 
\mathcal K_{C_0(\widehat A)}
\left(
\Ind_P^G\mathcal{H}_{M,\sigma}
\right)^{W(G,M,\sigma)}$. 
\end{corollary}

Combining Corollary \ref{cintocpminv} with Proposition \ref{pinjPhi}, we reach our main theorem.
\begin{theorem}\label{tmain1}
For a reductive group $G$, there is an injective $*$-homomorphism of $C^*$-algebras
\begin{equation}\label{emain1}
    C^*_{\mathcal{A}}(G(\mathbb{A}))\xrightarrow[]{}\bigoplus_{[M,\sigma]} \mathcal{K}_{C_{0}(\widehat{A_{M}})}(\Ind_{P}^{G}C_{0}(\widehat{A_{M}},H_{M,\sigma}))^{W(G,M,\sigma)},
\end{equation}
where the direct sum is taken over associate classes of the discrete Levi data of $G$. 
\end{theorem}

\section{The isomorphism theorem for $\GL(n)$}\label{sisoGLn}


Throughout this section, let $G=\GL(n)$.  Our goal is to prove that the homomorphism 
\begin{center}
    $C^{*}_{\mathcal{A}}(G(\mathbb{A}))\xrightarrow[]{\oplus \Phi_{M,\sigma}} 
\bigoplus_{[M,\sigma]}
C^*_{\mathcal A}(G(\mathbb{A}))_{M,\sigma}\to \bigoplus_{[M,\sigma]}\mathcal K_{C_0(\widehat{A_M})}
\left(
\Ind_P^G\mathcal{H}_{M,\sigma}
\right)^{W(G,M,\sigma)}$
\end{center}
is surjective, and hence an isomorphism of $C^*$-algebras by Proposition \ref{pinjPhi} and Corollary \ref{cintocpminv}. 

We first recall the automorphic representation theory of $\GL(n,\mathbb{A})$. 
The following result is due to Moeglin and Waldspurger (see \cite{MW1989}). 
\begin{theorem}\label{tMWGLndisc} 
The irreducible subrepresentations of  $L_{\mathrm{disc}}^{2}
\bigl(\GL(n,F)\backslash \GL(n,\mathbb{A})^{1}\bigr)$ have multiplicity one and are parametrized by pairs $(\sigma,p)$ where
\begin{itemize}
    \item $n=kp$ is divisible by $k$,
    \item $\sigma$ is an irreducible unitary cuspidal automorphic representation of $\GL(k,\mathbb{A})$. 
\end{itemize}
Suppose that $\pi$ is such an irreducible representation parameterized by $(\sigma,p)$ and $P$ is the standard parabolic subgroup of $\GL(n)$ of type
$(k,\ldots,k)$, and $\rho_{\sigma}$ is the non-tempered representation given by
\begin{center}
$(\sigma\otimes\cdots\otimes\sigma)\cdot\delta_{P}^{1/2}
  \colon m \longmapsto
  \sigma(m_{1})\lvert\det m_{1}\rvert^{\frac{p-1}{2}}
  \otimes\cdots\otimes
  \sigma(m_{p})\lvert\det m_{p}\rvert^{-\frac{p-1}{2}}$
\end{center}
of $M_{P}(\mathbb{A})\cong \prod_{j=1}^{p}\GL(k,\mathbb{A})$, 
then $\pi$ is the unique irreducible quotient of the induced
representation $\Ind_{P(\mathbb{A})}^{\GL(n,\mathbb{A})}(\rho_{\sigma})$.
\end{theorem}

For simplicity, we define the {\it Speh representation} of type $(\sigma,p)$, denoted by $\Speh(\sigma,p)$, to be a discrete automorphic representation with the parameter $(\sigma,p)$ in Theorem \ref{tMWGLndisc}, which is the unique irreducible quotient of
\begin{center}
$\Ind_{P(\mathbb{A})}^{G(\mathbb{A})}\bigl(
\sigma|\det(\cdot)|^{(p-1)/2}\otimes
\sigma|\det(\cdot)|^{(p-3)/2}\otimes
\cdots
\sigma|\det(\cdot)|^{(1-p)/2}
\bigr)$,
\end{center}
where the induction is taken to be normalized (see \cite[\S 10.8]{Getz300}).

Recall that for $\sigma\in\Pi_{\disc}\bigl(M(\mathbb A)^1\bigr)$ and $\lambda\in i\mathfrak{a}_M^*$, we let $\sigma_{\lambda}=\sigma\otimes e^{\langle \lambda,H_M(-)\rangle}$ and write $\Ind_{P(\mathbb A)}^{G(\mathbb A)}(\sigma_\lambda)$ simply as $\Ind_{P}^{G}(\sigma_\lambda)$. 
The following result is well known. 
We include a short proof for completeness.
\begin{lemma}\label{lGLnautoirrep}
Let $P=MN$ be  a parabolic subgroup of $\GL(n)$. 
For an irreducible unitary discrete automorphic representation $\sigma$ of $M(\mathbb{A})^1$ and $\lambda\in i\mathfrak{a}_M^*$, $\Ind_P^G(\sigma_{\lambda})$ is irreducible.    
\end{lemma}
\begin{proof}
We have $\Ind_{P(\mathbb{A})}^{G(\mathbb{A})}(\sigma_{\lambda})=\otimes'\Ind_{P(F_v)}^{G(F_v)}(\sigma_{\lambda})_v$. 
Then it follows from the fact that each local factor $\Ind_{P(F_v)}^{G(F_v)}(\sigma_{\lambda})_v$ is irreducible (see \cite[\S 13]{BaRe10} for archimedean places and \cite{Bernstein1984} for nonarchimedean places). 
\end{proof}

We now recall the discrete and full automorphic spectra of $\GL(n)$. 
Following \cite[\S 30, page 238]{Art05}, we let $\Psi_2(\GL(n))$ be the set of formal tensor products
\begin{center}
    $\psi=\sigma\boxtimes \nu_p$,
\end{center}
where $\sigma\in \Pi_{\cusp}(\GL(k,\mathbb{A})^1)$ with $n=kp$ and $\nu_p$ denotes the unique irreducible $p$-dimensional representation of $\SL(2,\mathbb{C})$. 
By Theorem \ref{tMWGLndisc} (and also \cite{JacShlk1981}), there is an association
\begin{equation}\label{eMWdiscGLn}
    \psi=\sigma\boxtimes \nu_p\longmapsto \Ind_P^G\bigl((\sigma\otimes\cdots\otimes\sigma)\cdot\delta_{P}^{1/2}\bigr).
\end{equation}
If we denote the unique irreducible quotient of $\Ind_P^G\bigl((\sigma\otimes\cdots\otimes\sigma)\cdot\delta_{P}^{1/2}\bigr)$ by $\pi_{\psi}$,
then
\begin{equation}\label{epsipipsi}
    \psi\mapsto \pi_{\psi}
\end{equation}
gives a bijection between $\Psi_2(\GL(n))$ and $\Pi_{\disc}(\GL(n,\mathbb{A})^1)$.

For the entire automorphic spectrum of $\GL(n)$ besides the discrete one, 
we let $\Psi(\GL(n))$ be the set of formal unordered direct sums
\begin{equation}\label{eArthurautopara}
    \psi=l_1\psi_1\boxplus\cdots\boxplus l_r\psi_r
\end{equation}
with $l_j$ being positive integers and $\psi_j=\sigma_j\boxtimes \nu_{p_j}\in \Psi_2(\GL(n_j))$ such that
\begin{center}
    $n=\sum_{j=1}^{r}l_j n_j=\sum_{j=1}^{r}l_j m_j p_j$. 
\end{center}
We let $P$ be the standard parabolic with its Levi component
\begin{center}
    $M=\prod_{j=1}^{r} \GL(n_j)^{l_j}$. 
\end{center}
We form the corresponding induced representation
\begin{equation}\label{eArthurasso}
    \pi_{\psi}:=\Ind_{P}^{G}\bigl(\otimes \pi_{\psi_j}^{\otimes l_j} \bigr)
\end{equation}
with $\pi_{\psi_j}$ defined in \eqref{epsipipsi}. 
By Lemma \ref{lGLnautoirrep}, $\pi_\psi$ is irreducible.  
Actually, the map 
\begin{center}
    $\psi\mapsto \pi_{\psi}$
\end{center}
gives a bijection from $\Psi(\GL(n))$ to $\Pi_{\mathcal{A}}(\GL(n,\mathbb{A}))$ (see \cite[Theorem 4.4]{JacShlk1981} and \cite[\S 30]{Art05}).

\subsection{Disjointness and surjectivity for $\GL(n)$}\label{ssGLnsurj}

We review some explicit results on the discrete automorphic representations of $\GL(n,\mathbb{A})$. 
Let $n=md$ and $P$ be a standard parabolic subgroup of $\GL(n)$ of type $(d,\cdots,d)$. 
Then the Levi factor $M_P$ is given as
\begin{center}
    $M_P\cong \prod_{k=1}^{m}\GL(d)$. 
\end{center}
Suppose $\delta$ is an irreducible cuspidal representation of $\GL(d,\mathbb{A})$. 
We take $\lambda=(\lambda_1,\cdots,\lambda_m)\in \mathfrak{a}_{M,\mathbb{C}}^{*}\cong \mathbb{C}^{m}$. 

The following result is due to \cite[Theorem 4.4]{JacShlk1981}. 

\begin{theorem}\label{tGLncuspdisj}
For $i=1,2$, let $P_i=M_i N_i$ be a parabolic subgroup of $\GL(n)$, $\lambda_i\in \mathfrak{a}_{M_i,\mathbb{C}}^*$ and $\delta_i\in \Pi_{\cusp}(M_{i}(\mathbb{A}))$. 
Suppose $\pi$ is an irreducible automorphic representation of $\GL(n,\mathbb{A})$ which is a
common irreducible constituent of $\Ind_{P_1(\mathbb{A})}^{\GL(n,\mathbb{A})}
\bigl(\delta_1\otimes e^{\lambda_1}\bigr)$ and $\Ind_{P_2(\mathbb{A})}^{\GL(n,\mathbb{A})}\bigl(\delta_2\otimes e^{\lambda_2}\bigr)$. 
Then $M_1,M_2$ are $\GL(n)$-conjugate and $(\delta_1,\lambda_1),(\delta_2,\lambda_2)$ are identical up to permutations of their factors. 
\end{theorem}

We give the disjointness of the discrete Levi datum.  
\begin{proposition}\label{pGLndiscdisj}
For $i=1,2$, let $P_i=M_i N_i$ be a parabolic subgroup of $\GL(n)$, $\lambda_i\in i\mathfrak{a}_{M_i}^*$ and $\sigma_i\in \Pi_{\disc}(M_{i}(\mathbb{A})^{1})$. 
If $\Ind_{P_1(\mathbb{A})}^{\GL(n,\mathbb{A})}
\bigl(\sigma_1\otimes e^{\lambda_1}\bigr)$ and $\Ind_{P_2(\mathbb{A})}^{\GL(n,\mathbb{A})}\bigl(\sigma_2\otimes e^{\lambda_2}\bigr)$ share  a common irreducible constituent,  
then there exists $w\in W$ such that 
\begin{center}
    $wM_1=M_2$, $w\sigma_1\cong \sigma_2$ and $w\lambda_1=\lambda_2$. 
\end{center}
\end{proposition}
\begin{proof}
  For $r\geq 1$, we simply write
$G_r=\GL(r)$, $A_r=A_{G_r}(\mathbb R)^0$ and the character $\nu_r(g)=|\det(g)|_{\mathbb A}$ of $G_r$. 
If $\pi$ is an irreducible representation of $G_r(\mathbb A)^1$, we denote by $\widetilde{\pi}$ its canonical extension to $G_r(\mathbb A)$ on which $A_r$ acts trivially. 

For $i=1,2$, we assume
\begin{center}
 $M_i=\prod_{j=1}^{r_i}G_{n_{i,j}}$ and  $\sigma_i\simeq\bigotimes_{j=1}^{r_i}\sigma_{i,j}$, 
\end{center}
where
\begin{center}
    $\sigma_{i,j}\in
\Pi_{\disc}\bigl(G_{n_{i,j}}(\mathbb A)^1\bigr)$. 
\end{center}
The restriction of $e^{\lambda_i}$ to the $j$-th factor of $M_i(\mathbb A)$
has the form $\nu_{n_{i,j}}^{\,it_{i,j}}$ for some $t_{i,j}\in\mathbb R$.

By applying Theorem \ref{tMWGLndisc} to $\sigma_{i,j}$, 
there are uniquely determined integers
$m_{i,j},d_{i,j}\geq 1$ and a cuspidal automorphic representation $\rho_{i,j}\in \Pi_{\cusp}(G_{m_{i,j}}(\mathbb A)^1)$ such that 
\begin{center}
    $n_{i,j}=m_{i,j}\cdot d_{i,j}$, and $\widetilde{\sigma_{i,j}}\cong
\Speh(\rho_{i,j},d_{i,j})$. 
\end{center}

Consequently, $\widetilde{\sigma_{i,j}}\otimes\nu_{n_{i,j}}^{\,it_{i,j}}$ is the unique irreducible quotient of the normalized parabolic induction from the following irreducible representation of $\prod_{k=1}^{d_{i,j}}G_{m_{i,j}}(\mathbb{A})$: 
\begin{center}
    $\rho_{i,j}
 \nu_{m_{i,j}}^{\frac{d_{i,j}-1}{2}+it_{i,j}}
\otimes
\rho_{i,j}
 \nu_{m_{i,j}}^{\frac{d_{i,j}-3}{2}+it_{i,j}}
\otimes\cdots\otimes
\rho_{i,j}
 \nu_{m_{i,j}}^{-\frac{d_{i,j}-1}{2}+it_{i,j}}$. 
\end{center}
Notice that the cuspidal representation $\widetilde{\rho_{i,j}}$ is not twisted.  

Let $\mathcal{C}_i$ be the multiset
\begin{center}
$\mathcal C_i
=
\bigsqcup_{j=1}^{r_i}
\left\{
\left(\rho_{i,j},
 \frac{d_{i,j}-1}{2}+it_{i,j}\right),
\left(\rho_{i,j},
 \frac{d_{i,j}-3}{2}+it_{i,j}\right),
\ldots,
\left(\rho_{i,j},
 -\frac{d_{i,j}-1}{2}+it_{i,j}\right)
\right\}$. 
\end{center}
By the defining quotient property of the Speh representations, together
with the transitivity of normalized parabolic induction, any
common irreducible constituent of
\begin{center}
    $\Ind_{P_1(\mathbb A)}^{G_n(\mathbb A)}
   (\sigma_1\otimes e^{\lambda_1})$ and $\Ind_{P_2(\mathbb A)}^{G_n(\mathbb A)}
   (\sigma_2\otimes e^{\lambda_2})$
\end{center}
is an irreducible subquotient of both cuspidally induced representations
determined by $\mathcal C_1$ and $\mathcal C_2$.
Thus, by Theorem \ref{tGLncuspdisj}, we have 
\begin{center}
    $\mathcal{C}_1=\mathcal{C}_2$ 
\end{center}
as multisets, up to permutations of the factors. 
It then remains to recover the original Speh blocks from this expanded cuspidal datum.  

Fix a cuspidal representation $\rho$ (and $\widetilde{\rho}$) and $t\in\mathbb R$.   For $k\in\frac12\mathbb Z_{\geq0}$, we let
\begin{center}
    $\mu_i(\rho,t,k):=$ the multiplicity of $(\rho,k+it)$ in $\mathcal C_i$
\end{center}
and let
\begin{center}
    $c_i(\rho,t,d): =
\#\left\{
j:
\rho_{i,j}\simeq\rho,\;
t_{i,j}=t,\;
d_{i,j}=d
\right\}$. 
\end{center}
Note that a Speh block of length $d'$ contributes the exponent $(d-1)/2+it$ precisely when $d'=d,d+2,d+4,\ldots$. 
Hence we obtain
\begin{center}
    $\mu_i\left(\rho,t,\frac{d-1}{2}\right)
=
\sum_{\ell\geq0}
c_i(\rho,t,d+2\ell)$ and $\mu_i\left(\rho,t,\frac{d+1}{2}\right)
=
\sum_{\ell\geq1}
c_i(\rho,t,d+2\ell)$. 
\end{center}
Therefore we obtain  
\begin{center}
    $c_i(\rho,t,d)
=
\mu_i\left(\rho,t,\frac{d-1}{2}\right)
-
\mu_i\left(\rho,t,\frac{d+1}{2}\right)$.
\end{center}
Since $\mathcal C_1=\mathcal C_2$, it follows that $c_1(\rho,t,d)=c_2(\rho,t,d)$
for every cuspidal representation $\rho$ (of some $G_r$), $t\in\mathbb{R}$, and $d\geq 1$. 
Thus the multisets of triples $\left\{
(\rho_{1,j},t_{1,j},d_{1,j})
\right\}_{j=1}^{r_1}$ and $\left\{
(\rho_{2,j},t_{2,j},d_{2,j})
\right\}_{j=1}^{r_2}$ coincide. 
In particular, after permuting the Levi blocks (or equivalently the $W$-action), we conclude that $r_1=r_2=r$, 
\begin{center}
    $n_{1,j}=n_{2,j}$ and $\sigma_{1,j}\otimes\nu_{n_{1,j}}^{it_{1,j}}
\cong
\sigma_{2,j}\otimes\nu_{n_{2,j}}^{it_{2,j}}$
\end{center}
for each $1\leq j\leq r_1=r_2$. 
Then we can compare the actions of $\GL(n_j,\mathbb{A})^1$ and $A_{\GL(n_j)}(\mathbb{R})^0$ separately, which implies $\sigma_{1,j}\cong \sigma_{2,j}$ and $t_{1,j}=t_{2,j}$. 
Hence we obtain 
\begin{center}
    $wM_1=M_2$, $w\sigma_1
\cong
\sigma_2$, and $w\lambda_1=\lambda_2$. 
\end{center}
for some $w\in W$
\end{proof}

\begin{remark}
Proposition \ref{pGLndiscdisj} is not generally true for an arbitrary reductive group $G$ or non-unitary representations of Levi subgroups. 
Let $Q=LU\subset M$ and $P=MN\subset G$. 
Let $R$ be the
parabolic subgroup of $G$ with Levi factor $L$ (see Example \ref{exmptransind}). 
Suppose that $\sigma\in \Pi_{\disc}(M(\mathbb A)^1)$
is non-cuspidal. Then there exist a cuspidal automorphic representation $\delta\in \Pi_{\cusp}(L(\mathbb A))$
and $\nu\in \mathfrak{a}_{L,\mathbb C}^{M,*}$ 
such that $\sigma$ is an irreducible subquotient of $\Ind_{Q(\mathbb A)}^{M(\mathbb A)}
\bigl(\delta\otimes e^\nu\bigr)$. 

Now let $\lambda\in i\mathfrak{a}_M^*$.  
By transitivity of normalized parabolic induction, $\Ind_{P(\mathbb A)}^{G(\mathbb A)}
\bigl(\sigma\otimes e^\lambda\bigr)$
is a subquotient of
\begin{center}
   $\Ind_{P(\mathbb A)}^{G(\mathbb A)}
\left(
\Ind_{Q(\mathbb A)}^{M(\mathbb A)}
(\delta\otimes e^\nu)
\otimes e^\lambda
\right)$, 
\end{center}
which is naturally isomorphic to
\begin{center}
    $\Ind_{R(\mathbb A)}^{G(\mathbb A)}
\bigl(
\delta\otimes e^{\nu+\iota(\lambda)}
\bigr)$. 
\end{center}
Here $\iota:\mathfrak{a}_M^*\longrightarrow \mathfrak{a}_L^*$ is the natural pullback and $e^{\nu+\iota(\lambda)}$ may fail to be unitary. 
Hence $\Ind_{P(\mathbb A)}^{G(\mathbb A)}
\bigl(\sigma\otimes e^\lambda\bigr)$ and $\Ind_{R(\mathbb A)}^{G(\mathbb A)}
\bigl(
\delta\otimes e^{\nu+\iota(\lambda)}
\bigr)$ 
share an irreducible constituent, although the pairs $(M,\sigma)$ and $(L,\delta)$ 
are not associated. 
\end{remark}

\begin{corollary}\label{csurjalongMsigma}
For $G=\GL(n)$, the $C^*$-algebra homomorphism
\begin{center}
$\Phi=\bigoplus_{[M,\sigma]}\Phi_{M,\sigma}\colon C^{*}_{\mathcal{A}}(G(\mathbb{A}))
\longrightarrow
\bigoplus_{[M,\sigma]}
C^*_{\mathcal A}(G(\mathbb{A}))_{M,\sigma}$
\end{center}
is an isomorphism.
\end{corollary}
\begin{proof}
By Proposition \ref{pinjPhi}, it remains to show surjectivity. 
By Lemma \ref{lCAGfactor} and Proposition \ref{pGLndiscdisj}, we know that the supports of $\Phi_{M,\sigma}$ are pairwise disjoint. 
Then, by \cite[Lemma 5.14]{CCH2016} and Proposition \ref{pweldefhom}, $\Phi$ is surjective. 
\end{proof}

\subsection{The structure of the component algebra for $\GL(n)$}
\label{ssGLncomponent}

The following theorem gives the uniqueness of the unitary induction parameter for a fixed discrete datum $\sigma$.  

\begin{theorem}
\label{tuniquesigmaparaGLn}
Let $\sigma\in\Pi_{\disc}\bigl(M(\mathbb A)^1\bigr)$ and $\lambda_1,\lambda_2\in i\mathfrak{a}_M^*$. 
We have
\begin{center}
    $\Ind_P^G(\sigma_{\lambda_1})
\cong
\Ind_P^G(\sigma_{\lambda_2})$
\end{center}
if and only if $\lambda_2=w\lambda_1$ for some $w\in W(G,M,\sigma)$. 
\end{theorem}

\begin{proof}
Suppose first that $\Ind_P^G(\sigma_{\lambda_1})
\cong
\Ind_P^G(\sigma_{\lambda_2})$. 
Then the two induced representations have a common irreducible
constituent. 
By Proposition \ref{pGLndiscdisj}, there exists $w\in W(G,M)$ such that 
\begin{center}
    $w\sigma\simeq\sigma$ and $w\lambda_1=\lambda_2$. 
\end{center}
Thus, $w\lambda_1=\lambda_2$ for $w\in W(G,M,\sigma)$. 

Conversely, suppose that $\lambda_2=w\lambda_1$ for some $w\in W(G,M,\sigma)$. 
Then $w\sigma\simeq\sigma$ and therefore $w(\sigma_{\lambda_1})\cong\sigma_{w\lambda_1}=\sigma_{\lambda_2}$. 
Since $\lambda_1\in i\mathfrak a_M^*$, the normalized standard intertwining operator associated with $w$ is unitary and invertible.
It consequently induces an isomorphism $\Ind_P^G(\sigma_{\lambda_1})
\cong
\Ind_P^G(\sigma_{\lambda_2})$. 
This proves the assertion.
\end{proof}

For $\varphi\in \widehat{A}$, we consider the group
\begin{center}
    $W(\sigma,\varphi):=\{w\in W(G,M,\sigma)\mid w\varphi=\varphi\}$. 
\end{center}
By Proposition \ref{pGlobalWeylAction}, we know that
\begin{center}
    $U_w\in \Aut(\Ind_P^G(\sigma\otimes\varphi))$
\end{center}
for $w\in W(\sigma,\varphi)$. 
It holds for general reductive groups, not only for $\GL(n)$. 
\begin{lemma}\label{ltricommalg}
The $C^*$-algebra generated by $\{U_w\mid w\in W(\sigma,\varphi)\}$ is $\mathbb{C}$.
\end{lemma}
\begin{proof}
By Lemma \ref{lGLnautoirrep}, we know $\End_{G}(\Ind_P^G(\sigma\otimes\varphi))\cong \mathbb{C}$.  
It then follows from Proposition \ref{pGlobalWeylAction}.  
\end{proof}

\begin{proposition}\label{psurjGLncomp}
For $G=\GL(n)$, $C^*_{\mathcal{A}}(G)_{M,\sigma}\cong \mathcal{K}(\Ind_P^G \mathcal{H}_{M,\sigma})^{W(G,M,\sigma)}$.  
\end{proposition}
\begin{proof}
Write $\mathcal{K}_{C_0(\widehat A)}(\mathcal{E}_{M,\sigma})^{W(G,M,\sigma)}$ simply as $B_{M,\sigma}$. 
By Corollary \ref{cintocpminv}, we already have an inclusion
\begin{equation}\label{eABinclGLn}
C^*_{\mathcal A}(G)_{M,\sigma}
\subseteq
B_{M,\sigma}.
\end{equation}
It suffices to show this inclusion is surjective.

Since $G=\GL(n)$, every Levi subgroup is a product of general
linear groups.  
By Theorem \ref{tMWGLndisc}, 
the discrete representation $\sigma$ of $M$ occurs with
multiplicity one. 
Therefore the fiber of $\mathcal{E}_{M,\sigma}$ over $\varphi\in\widehat A$ is the Hilbert space of the induced representation $\Ind_P^G(\sigma\otimes\varphi)$. 
By Lemma \ref{ltricommalg}, we have already shown  
\begin{center}
    $\mathfrak{I}(\sigma,\varphi)
:=
C^*\bigl(
U_{w,\varphi}:w\in W(\sigma,\varphi)
\bigr)
\subseteq B(\Ind_P^G(\sigma\otimes\varphi))$
\end{center}
is trivial for any $\varphi\in\widehat{A}$. 
Thus the only non-zero projection in $\mathfrak{I}(\sigma,\varphi)$ is $1$. 
Applying the fixed-point-algebra argument of
\cite[Lemma~6.5]{CCH2016} to the finite $W_{\sigma}$-action of
Proposition \ref{pGlobalWeylAction}, we obtain:  
\begin{enumerate}
\item every irreducible representation of
$B_{M,\sigma}$ is equivalent to the representation
\begin{center}
    $e_{\varphi}\colon B_{M,\sigma}\to \mathcal{K}(\Ind_P^G(\sigma\otimes\varphi))$
\end{center} 
obtained by evaluation at some
$\varphi\in\widehat A$;

\item for $\varphi_1,\varphi_2\in\widehat A$, Theorem \ref{tuniquesigmaparaGLn} implies that 
\begin{equation}\label{eBirrepeqGLn}
    e_{\varphi_1}\cong e_{\varphi_2} \text{ if and only if  }\varphi_2=w\varphi_1 \text{ for some }w\in W(G,M,\sigma). 
\end{equation}

\item $B_{M,\sigma}$ is postliminal (see \cite[5.4.13]{DiCalg}).
\end{enumerate}

We next examine the restrictions of these irreducible
representations to the subalgebra
$C^*_{\mathcal A}(G)_{M,\sigma}$.
Consider the restriction $r_{\varphi}:=e_{\varphi}\mid_{C^*_{\mathcal A}(G)_{M,\sigma}}$. 
By definition, the pullback $r_{\varphi}\circ \Phi_{M,\sigma}$ 
is precisely the integrated representation of $C^*_{\mathcal A}(G(\mathbb A))$ associated with $\Ind_P^G(\sigma\otimes\varphi)$, which is irreducible by Lemma \ref{lGLnautoirrep}. 
Hence $r_{\varphi}$ is irreducible.  
If $V$ is invariant under $r_{\varphi}(C^*_{\mathcal A}(G)_{M,\sigma})$, then $V$ is invariant under
$(r_{\varphi}\circ \Phi_{M,\sigma})(C^*_{\mathcal A}(G))$. 
Since the latter representation is irreducible, $V$ is either trivial or the entire space. 

It remains to verify that inequivalent irreducible representations of
$B_{M,\sigma}$ remain inequivalent after restriction to
$C^*_{\mathcal A}(G)_{M,\sigma}$. 
Suppose that 
\begin{center}
    $r_{\varphi_1}\cong r_{\varphi_2}$
\end{center}
as representations of $C_{\mathcal{A}}^{*}(G(\mathbb{A}))_{M,\sigma}$. 
After composing with $\Phi_{M,\sigma}$, we obtain equivalent
representations of $C^*_{\mathcal A}(G(\mathbb A))$, and hence
equivalent $G(\mathbb A)$-representations
$\Ind_P^G(\sigma\otimes\varphi_1)$ and $
\Ind_P^G(\sigma\otimes\varphi_2)$. 
By Theorem \ref{tuniquesigmaparaGLn}, there exists
$w\in W(G,M,\sigma)=W_{\sigma}$ such that $\varphi_2=w\varphi_1$. 
It then follows from \eqref{eBirrepeqGLn} that $e_{\varphi_1}\cong e_{\varphi_2}$. 
Thus inequivalent irreducible representations of
$B_{M,\sigma}$ have inequivalent restrictions to
$C^*_{\mathcal A}(G)_{M,\sigma}$. 
By the noncommutative Stone--Weierstrass theorem \cite[Theorem 11.1.8]{DiCalg}, we conclude that $C^*_{\mathcal A}(G(\mathbb{A}))_{M,\sigma}=B_{M,\sigma}$, which proves the proposition.
\end{proof}

To indicate the dependence on the Levi subgroup, we write $A_M$ for the connected component of the real points of the central subgroup of a Levi subgroup $M$ of $G$. 
Note that the $\sigma$-isotypic subspace $H_{M,\sigma}$ reduces to $H_{\sigma}$ by the multiplicity-one theorem for a general linear group. 
\begin{corollary}\label{cisoGLn}
    $C^*_{\mathcal A}(\GL(n,\mathbb{A}))\cong \bigoplus_{[M,\sigma]}\mathcal{K}\bigl(\Ind_P^G( C_0(\widehat{A_M},H_{\sigma}))\bigr)^{W(G,M,\sigma)}$
\end{corollary}
\begin{proof}
It follows from Corollary \ref{csurjalongMsigma} and Proposition \ref{psurjGLncomp}. 
\end{proof}

\section{The isomorphism theorem for inner forms of $\GL(n)$}
\label{sisoinnerGLn}

This section is devoted to the analogue of Corollary \ref{cisoGLn} for inner forms of $\GL(n)$. 
Let $D$ be a central division algebra over $\mathbb{Q}$ of degree $d$, i.e., $\dim_{\mathbb{Q}}D=d^2$, and let
\begin{center}
$G=G_r:=\GL(r,D)$ 
\end{center}
with $n=rd$. 
It is known that $G$ is an inner form of $\GL(n,\mathbb{Q})$ and any inner form of $\GL(n)$ is given by such a form. 
Observe that $G(\mathbb{A})\cong \GL(r,D(\mathbb{A}))$. 
For $g\in G(\mathbb{A})$, we write
\begin{center}
   $\eta(g)=|\Nrd(g)|_{\mathbb A}$ 
\end{center}
for the adelic absolute value of the reduced norm on $G(\mathbb{A})$. 

Recall that the global Jacquet--Langlands correspondence \cite{Badu08,BaRe10} gives an injective map
\begin{center}
    $\JL\colon \Pi_{\disc}(\GL(r,D(\mathbb{A})))\to \Pi_{\disc}(\GL(n,\mathbb{A}))$. 
\end{center}
By Theorem \ref{tMWGLndisc}, for $\pi\in \Pi_{\disc}(G_r(\mathbb{A}))$, we know that 
\begin{center}
    $\JL(\pi)\cong \Speh(\sigma,m)$
\end{center}
for a unique $\sigma\in \Pi_{\cusp}(\GL(l,\mathbb{A}))$ and the factorization $n=ml$.  
Following \cite[\S 18]{BaRe10}, we let $m(\pi)=m$ and define the character on $\GL(r,D(\mathbb{A}))$ by
\begin{center}
$\eta_{\pi}:=\eta^{m(\pi)}=|\Nrd(\cdot)|_{\mathbb{A}}^m$. 
\end{center}
Then we define the representation $\MW(\pi,k)$ to be the unique irreducible quotient of the induced representation
\begin{equation}\label{einnerMW} \Ind_{P(\mathbb{A})}^{\GL(rk,D(\mathbb{A}))}(\pi\eta_{\pi}^{(k-1)/2}\otimes \pi\eta_{\pi}^{(k-3)/2}\otimes\cdots\otimes\pi\eta_{\pi}^{(1-k)/2}), 
\end{equation}
where $P$ is the standard parabolic subgroup of $\GL(rk,D(\mathbb{A}))$ with $M_P\cong \prod_{i=1}^{k}\GL(r,D(\mathbb{A}))$ being its Levi factor. 

The following result is due to Badulescu and Renard (see \cite[Proposition 18.2]{BaRe10} and \cite[Theorem 5.1]{Badu08}).

\begin{theorem}\label{tInnerMWdisc}
The irreducible subrepresentations of $L_{\mathrm{disc}}^{2}
\bigl(\GL(r,D(\mathbb{Q}))\backslash \GL(r,D(\mathbb{A}))^{1}\bigr)$ 
have multiplicity one and are parametrized by pairs $(k,\sigma)$ where
\begin{itemize}
    \item $r=kp$ is divisible by $k$,
    \item $\sigma$ is an irreducible unitary cuspidal automorphic representation of $\GL(k,D(\mathbb{A}))$. 
\end{itemize}
Suppose that $\pi$ is parametrized by $(k,\sigma)$ and $P$ is the standard parabolic subgroup of $\GL(r,D(\mathbb{A}))$ of type
$(k,\ldots,k)$, 
then
\begin{center}
    $\pi\cong \MW(\sigma,p)$. 
\end{center}
\end{theorem}

We shall also use the following  irreducibility property as an analogue of Lemma \ref{lGLnautoirrep}. 
Let $P=MN$ be a parabolic subgroup of $\GL(r,D)$. 
Note that $\Ind_{P(\mathbb A)}^{G(\mathbb A)}(\sigma_\lambda)$, or simply $\Ind_{P}^{G}(\sigma_\lambda)$, is the parabolic induction of $\sigma_{\lambda}=\sigma\otimes e^{\langle \lambda,H_M(-)\rangle}$ with
$\sigma\in\Pi_{\disc}\bigl(M(\mathbb A)^1\bigr)$ and $\lambda\in i\mathfrak{a}_M^*$. 

\begin{lemma}\label{lInnerGLnautoirrep}
Let $P=MN$ be a parabolic subgroup of $G=\GL(r,D)$,
$\sigma\in\Pi_{\disc}(M(\mathbb A)^1)$ and
$\lambda\in i\mathfrak a_M^*$.  
Then $\Ind_P^G(\sigma_{\lambda})$ is irreducible. 
\end{lemma}
\begin{proof}
Consider the restricted tensor product $\Ind_P^G(\sigma_\lambda)
\simeq
\bigotimes_v
\Ind_{P(F_v)}^{G(F_v)}(\sigma_{\lambda,v})$, 
where each local induced representation is irreducible and unitary. 
It remains to prove each local irreducibility. 
For the archimedean places, it follows from \cite[\S 12]{BaRe10}. 
For nonarchimedean places, it follows from \cite{Secherre2009} (see also \cite{Bernstein1984}). 
\end{proof}

We have the following disjointness result for cuspidal Levi data of $\GL(r,D)$ (see \cite[Proposition 1.6, Corollary 1.7]{BaRe10}). 
Please note that they consider the so-called essentially cuspidal representations, which are twists of unitary cuspidal representations by a real power of $\eta$. 

\begin{theorem}\label{tInnerGLncuspsupport}
For $i=1,2$, let $P_i=M_i N_i$ be a parabolic subgroup of $\GL(r,D)$, $\lambda_i\in \mathfrak{a}_{M_i,\mathbb{C}}^*$ and $\delta_i\in \Pi_{\cusp}(M_{i}(\mathbb{A})^1)$. 
Suppose $\pi$ is an irreducible automorphic representation of $\GL(r,D(\mathbb{A}))$ which is a
common irreducible constituent of $\Ind_{P_1(\mathbb{A})}^{\GL(r,D(\mathbb{A}))}
\bigl(\delta_1\otimes e^{\lambda_1}\bigr)$ and $\Ind_{P_2(\mathbb{A})}^{\GL(r,D(\mathbb{A}))}\bigl(\delta_2\otimes e^{\lambda_2}\bigr)$. 
Then $M_1,M_2$ together with $(\delta_1,\lambda_1),(\delta_2,\lambda_2)$ are $\GL(r,D)$-conjugate.   
\end{theorem}

We have the disjointness property for discrete Levi datum of $\GL(r,D)$ as well, which is the analogue of Proposition~\ref{pGLndiscdisj}.

\begin{proposition}\label{pInnerGLndiscdisj}
For $i=1,2$, let $P_i=M_i N_i$ be a parabolic subgroup of $\GL(r,D)$, $\lambda_i\in i\mathfrak{a}_{M_i}^*$ and $\sigma_i\in \Pi_{\disc}(M_{i}(\mathbb{A})^{1})$. 
If $\Ind_{P_1(\mathbb{A})}^{\GL(r,D(\mathbb{A}))}
\bigl(\sigma_1\otimes e^{\lambda_1}\bigr)$ and $\Ind_{P_2(\mathbb{A})}^{\GL(r,D(\mathbb{A}))}\bigl(\sigma_2\otimes e^{\lambda_2}\bigr)$ share  a common irreducible constituent,  
then there exists $w\in W$ such that 
\begin{center}
    $wM_1=M_2$, $w\sigma_1\cong \sigma_2$ and $w\lambda_1=\lambda_2$. 
\end{center}
\end{proposition}
\begin{proof}
For $s\geq 1$, we simply write $A_s=A_{G_s}(\mathbb R)^0$.  
If $\pi$ is an irreducible representation of $G_s(\mathbb A)^1$, we denote by $\widetilde{\pi}$ its canonical extension to $G_s(\mathbb A)$ on which $A_s$ acts trivially. 
For $i=1,2$, we assume
\begin{center}
 $M_i=\prod_{j=1}^{r_i}G_{s_{i,j}}=\prod_{j=1}^{r_i}\GL(s_{i,j},D)$ and  $\sigma_i\simeq\bigotimes_{j=1}^{r_i}\sigma_{i,j}$, 
\end{center}
where $\sigma_{i,j}\in
\Pi_{\disc}\bigl(G_{s_{i,j}}(\mathbb A)^1\bigr)$. 
The restriction of $e^{\lambda_i}$ to the $j$-th factor of $M_i(\mathbb A)$
has the form $\eta^{it_{i,j}}$ for some $t_{i,j}\in\mathbb R$.

By applying Theorem \ref{tInnerMWdisc} to $\sigma_{i,j}$, 
there are uniquely determined integers
$m_{i,j},d_{i,j}\geq 1$ and a cuspidal automorphic representation $\rho_{i,j}\in \Pi_{\cusp}(G_{m_{i,j}}(\mathbb A)^1)$ such that 
\begin{center}
    $s_{i,j}=m_{i,j}\cdot d_{i,j}$, and $\widetilde{\sigma_{i,j}}\cong
\MW(\rho_{i,j},d_{i,j})$. 
\end{center}
Consequently, $\widetilde{\sigma_{i,j}}\otimes\eta^{\,it_{i,j}}$ is the unique irreducible quotient of the normalized parabolic induction from the following irreducible representation of $\prod_{k=1}^{d_{i,j}}G_{m_{i,j}}(\mathbb{A})$: 
\begin{center}
    $\rho_{i,j}
 \eta_{i,j}^{\frac{d_{i,j}-1}{2}}\eta^{it_{i,j}}
\otimes
\rho_{i,j}
 \eta_{i,j}^{\frac{d_{i,j}-3}{2}}\eta^{it_{i,j}}
\otimes\cdots\otimes
\rho_{i,j}
 \eta_{i,j}^{-\frac{d_{i,j}-1}{2}}\eta^{it_{i,j}}$,
\end{center}
where $\eta_{i,j}=\eta_{\rho_{i,j}}=\eta^{m(\rho_{i,j})}$ (see \eqref{einnerMW}). 
We let $\mathcal{C}_i$ be the following multiset: 
\begin{center}
$\bigsqcup_{j=1}^{r_i}
\left\{
\left(\rho_{i,j},
 m(\rho_{i,j})\frac{d_{i,j}-1}{2}+it_{i,j}\right),
\left(\rho_{i,j},
 m(\rho_{i,j})\frac{d_{i,j}-3}{2}+it_{i,j}\right),
\ldots,
\left(\rho_{i,j},
 -m(\rho_{i,j})\frac{d_{i,j}-1}{2}+it_{i,j}\right)
\right\}$. 
\end{center}
By the defining quotient property of the representation $\MW(\sigma,p)$ (see Theorem \ref{tInnerMWdisc}), together
with the transitivity of normalized parabolic induction, any
common irreducible constituent of
\begin{center}
    $\Ind_{P_1(\mathbb A)}^{G_r(\mathbb A)}
   (\sigma_1\otimes e^{\lambda_1})$ and $\Ind_{P_2(\mathbb A)}^{G_r(\mathbb A)}
   (\sigma_2\otimes e^{\lambda_2})$
\end{center}
is an irreducible subquotient of both cuspidally induced representations
determined by $\mathcal C_1$ and $\mathcal C_2$.
Thus, by Theorem \ref{tInnerGLncuspsupport}, we have 
\begin{center}
    $\mathcal{C}_1=\mathcal{C}_2$
\end{center}
as multisets, up to permutations of the factors. 

It remains to recover the original $\MW(\cdot,\cdot)$ blocks from this expanded cuspidal datum.  
The proof then follows similarly from the proof of Proposition \ref{pGLndiscdisj}.  
Hence we have 
\begin{center}
    $wM_1=M_2$ and $w\bigl(\sigma_1\otimes e^{\lambda_1}\bigr)
\cong
\sigma_2\otimes e^{\lambda_2}$ 
\end{center}
for some $w\in W$.  
\end{proof}

The following result gives the uniqueness of the unitary induction parameter for a fixed discrete datum $\sigma$ of $\GL(r,D)$. 
It is the inner form version of Theorem \ref{tuniquesigmaparaGLn} and the proof follows similarly.

\begin{corollary}\label{cuniquesigmaparaInnerGLn}
Let $G=\GL(r,D)$ and $P=MN$ be a parabolic subgroup. 
Let $\sigma\in\Pi_{\disc}(M(\mathbb A)^1)$ and $\lambda_1,\lambda_2\in i\mathfrak a_M^*$. 
Then $\Ind_P^G(\sigma_{\lambda_1})\cong \Ind_P^G(\sigma_{\lambda_2})$ if and only if $\lambda_2=w\lambda_1$  
for some $ w\in W(G,M,\sigma)$.
\end{corollary}

\begin{corollary}\label{csurjInnerGLn}
For $G=\GL(r,D)$, the $C^*$-algebra homomorphism
\[
\Phi=
\bigoplus_{[M,\sigma]}\Phi_{M,\sigma}
\colon
C^*_{\mathcal A}(G(\mathbb A))
\longrightarrow
\bigoplus_{[M,\sigma]}
C^*_{\mathcal A}(G(\mathbb A))_{M,\sigma}
\]
is an isomorphism.
\end{corollary}
\begin{proof}
By Proposition \ref{pinjPhi}, it remains to show surjectivity. 
By Lemma \ref{lCAGfactor} and Proposition \ref{pInnerGLndiscdisj}, we know that the supports of $\Phi_{M,\sigma}$ are pairwise disjoint. 
Then, by \cite[Lemma 5.14]{CCH2016} and Proposition \ref{pweldefhom}, $\Phi$ is surjective. 
\end{proof}

Recall that $W(\sigma,\varphi)=\left\{
w\in W(G,M,\sigma):w\varphi=\varphi
\right\}$ for $\varphi\in\widehat A$ and there is a $G$-intertwining unitary operator 
\begin{center}
    $U_w\in \Aut(\Ind_P^G(\sigma\otimes\varphi))$
\end{center}
for $w\in W(\sigma,\varphi)$ by Proposition \ref{pGlobalWeylAction}. 

\begin{corollary}\label{cInnertricommalg}
For $G=\GL(r,D)$, the $C^*$-algebra generated by $\{U_w:w\in W(\sigma,\varphi)\}$ in $B\bigl(\Ind_P^G(\sigma\otimes\varphi)\bigr)$ is $\mathbb C$.
\end{corollary}

\begin{proof}
It follows from Lemma~\ref{lInnerGLnautoirrep}. 
\end{proof}

\begin{proposition}\label{psurjInnerGLncomp}
Let $G=\GL(r,D)$.  Then
\[
C^*_{\mathcal A}(G(\mathbb A))_{M,\sigma}
\cong
\mathcal K\bigl(\Ind_P^G\mathcal H_{M,\sigma}\bigr)
^{W(G,M,\sigma)}.
\]
\end{proposition}
\begin{proof}
The proof follows similarly from the proof of Proposition \ref{psurjGLncomp}. 
\end{proof}

As in Section \ref{ssGLncomponent}, we write $A_M$ instead of $A$ to distinguish the Levi subgroups $M$. 
By the multiplicity-one theorem for $\GL(r,D)$ (see Theorem \ref{tInnerMWdisc}), the $\sigma$-isotypic space $H_{M,\sigma}$ also reduces to a single copy of $H_\sigma$.

\begin{corollary}\label{cisoInnerGLn}
Let $G=\GL(r,D)$.  Then
\[
C^*_{\mathcal A}(G(\mathbb A))
\cong
\bigoplus_{[M,\sigma]}
\mathcal K
\left(
\Ind_P^G
\bigl(
C_0(\widehat{A_M},H_\sigma)
\bigr)
\right)^{W(G,M,\sigma)}.
\]
\end{corollary}

\begin{proof}
This follows immediately from
Corollary~\ref{csurjInnerGLn} and
Proposition~\ref{psurjInnerGLncomp}.
\end{proof}

\bibliographystyle{abbrv}
\typeout{}
\bibliography{MyLibrary} 

\end{document}